\documentclass{amsart}
\usepackage[T1]{fontenc}

\usepackage{amsmath}
\usepackage{amssymb}
\usepackage{mathrsfs}

\usepackage{float}
\usepackage{graphicx,color}
\usepackage{ascmac}

\usepackage{tabularx}

\usepackage{tikz}
\usetikzlibrary{calc,decorations.markings}
\usepackage{amsthm}
\theoremstyle{definition}
\newtheorem{THM}{Theorem}
\newtheorem{LEM}[THM]{Lemma}
\newtheorem{PROP}[THM]{Proposition}

\newtheorem{DEF}[THM]{Definition}

\newtheorem*{THM*}{Theorem}
\newtheorem*{LEM*}{Lemma}
\newtheorem*{PROP*}{Proposition}
\newtheorem*{COR*}{Corollary}
\newtheorem*{DEF*}{Definition}
\newtheorem*{RMK*}{Remark}
\newtheorem*{EX*}{Example}

\numberwithin{figure}{section}
\numberwithin{equation}{section}
\numberwithin{THM}{section}

\title{$A_2$ Skein Representations of Pure Braid Groups}
\author{Wataru Yuasa}
\address{Department of Mathematics\\
  Tokyo Institute of Technology\\
  2-12-1 Ookayama, Meguro-ku, Tokyo 152-8551, Japan}
\email[]{yuasa.w.aa@m.titech.ac.jp}
\subjclass[2010]{20F36, 57M99}
\keywords{Pure braid group; skein theory; $A_2$ web space.}

\begin{document}
%%%%%%%%%%%%%%%%%%%%
%%%%% abstract %%%%%
%%%%%%%%%%%%%%%%%%%%
\begin{abstract}
We define a family of representations $\{\rho_n\}_{n\geq 0}$ of a pure braid group $P_{2k}$.
These representations are obtained from an action of $P_{2k}$ on a certain type of $A_2$ web space with color $n$. 
The $A_2$ web space is a generalization of the Kauffman bracket skein module of a disk with marked points on its boundary. 
We also introduce a triangle-free basis of such an $A_2$ web space and calculate matrix representations of $\rho_n$ about the standard generators of $P_{2k}$.
\end{abstract}
\maketitle
%\tableofcontents
%\thispagestyle{empty}
%\pagestyle{empty}
%%%%%%%%%%%%%%%%%%%%%%%%
%%%%% middle arrow %%%%%
%%%%%%%%%%%%%%%%%%%%%%%%
\tikzset{->-/.style={decoration={
  markings,
  mark=at position #1 with {\arrow[black,thin]{>}}},postaction={decorate}}}
\tikzset{-<-/.style={decoration={
  markings,
  mark=at position #1 with {\arrow[black,thin]{<}}},postaction={decorate}}}
\tikzset{-|-/.style={decoration={
  markings,
  mark=at position #1 with {\arrow[black,thin]{|}}},postaction={decorate}}}
%%%%%%%%%%%%%%%%%%%%%%
%%%% triple line %%%%%
%%%%%%%%%%%%%%%%%%%%%%
\tikzset{
    triple/.style args={[#1] in [#2] in [#3]}{
        #1,preaction={preaction={draw,#3},draw,#2}
    }
}

%%%%%%%%%%%%%%%%%%%%%
%%%%% body text %%%%%
%%%%%%%%%%%%%%%%%%%%%
\section{Introduction}
The theory of quantum invariants has been developing from the discovery of the Jones polynomials of knots and links \cite{Jones85}. 
The Jones polynomial was reformulated through the linear skein theory, the Kauffman bracket and the Kauffman bracket skein module, by Kauffman~\cite{Kauffman87}. 
The linear skein theory is used to define and calculate many quantum invariants and quantum representations.
The quantum representation of mapping class groups of surfaces are introduced in \cite{BlanchetHabeggerMasbaumVogel95} and \cite{Roberts94}. 
Especially, 
Roberts defined it by using the linear skein theory of $3$-dimensional handlebodies \cite{Lickorish93B} and the recoupling theory, for example, see \cite{KauffmanLins94}. 
It is useful for computation of the quantum representation to translate skein elements into trivalent graphs with colored edges. 
The trivalent graph notation derives many useful formulae of the Kauffman bracket and makes it easy for us to calculate the quantum representations. 
Masbaum~\cite{Masbaum17} proved that the normal closure of powers of half-twists has infinite index in mapping class groups of punctured spheres by the quantum representation of braid groups. 
He used elementary skein theory whereas Stylianakis~\cite{Stylianakis15} used the Jones representation~\cite{Jones87}. 

In this paper, 
we introduce quantum representations of a pure braid group $P_{2k}$ of $2k$ strands by using skein theory of the $A_2$ bracket and the $A_2$ web spaces introduced by Kuperberg~\cite{Kuperberg96}. 
In the case of $A_1$, 
the skein theory corresponds to the theory of the Kauffman bracket and the Kauffman bracket skein modules. 
Our quantum representations are obtained by action of $P_{2k}$ on a clasped $A_2$ web space $W_{2k}(n^{\pm})$ and as a family $\{\rho_n\}_{n\geq 0}$ colored by non-negative integers. 
Moreover, 
we define a special basis called {\em triangle-free basis} of $W_{2k}(n^{\pm})$ by using trivalent graph notation of the $A_2$ webs. 
This basis is given with each choice of a colored triangle-free triangulation~\cite{AdinFirerRoichman10} of a $2k$-gon.
For a certain colored triangle-free triangulation $T_0$, 
we completely calculate the values of the matrix representation $\rho_n^{T_0}$ on the standard generators of $P_{2k}$ by using the $A_2$ skein theory. 

The paper is organized as follows. 
In Section~2, 
we introduce the $A_2$ web spaces and review known formulae of the $A_2$ bracket and new formulae used in the following sections. 
In Section~3, 
we consider a clasped $A_2$ web $W_{2k}(n^{\pm})$ on which we will act $P_{2k}$.
For each colored triangle-free triangulation of $2k$-gon, 
we give an ordered basis and we call it the triangle-free basis of $W_{2k}(n^{\pm})$.
In Section~4, 
we define the quantum representation $\rho_{n}\colon P_{2k}\to GL(W_{2k}(n^{\pm}))$. 
This representation has a natural matrix representation with respect to a triangle-free basis of $W_{2k}(n^{\pm}$. 
Finally, 
we compute the matrix representation $\rho_{n}^{T_0}$ for a special colored triangle-free triangulation $T_0$, 
but it enables to calculate matrix representations of every triangle-free basis. 
In this computation, 
we completely give a matrix of $\rho_{n}^{T_0}(A_{i,j})$ for any standard generator $A_{i,j}$ of $P_{2k}$.

\section{The $A_2$ web space}
We consider a unit disk $D$, a set of marked points $P=\{\,p_1, p_2, \dots, p_{k}\,\}$ on its boundary and a set of signs $\epsilon=\{\,\epsilon_1,\epsilon_2,\dots,\epsilon_{k}\,\}$. 
The marked points are arranged cyclically $p_1<p_2<\dots<p_k<p_1$ with respect to the orientation of $\partial D$ and each $p_i$ labeled by $\epsilon_i={+}\text{ or }{-}$. 
Let us denote such triple $(D,P,\epsilon)$ by $D(\epsilon_1,\dots,\epsilon_{k})$ or $D_{\epsilon}$.
A {\em bipartite uni-trivalent graph} $G$ is a directed graph such that every vertex is either trivalent or univalent and these vertices are divided into sinks or sources. 
A sink (resp. source) is a vertex such that all edges adjoining to the vertex point into (resp. away from) it.
A {\em bipartite trivalent graph} $G$ in $D_{\epsilon}$ is an embedding of a uni-trivalent graph into $D_{\epsilon}$ such that for any vertex $v$ has the following neighborhoods:
\begin{itemize}
\item \,\tikz[baseline=-.6ex]{
\draw [thin, dashed, fill=white] (0,0) circle [radius=.5];
\draw[-<-=.5] (0:0) -- (90:.5); 
\draw[-<-=.5] (0:0) -- (210:.5); 
\draw[-<-=.5] (0:0) -- (-30:.5);
\node (v) at (0,0) [above right]{$v$};
\fill (0,0) circle [radius=1pt];
}\, 
or 
\,\tikz[baseline=-.6ex]{
\draw [thin, dashed] (0,0) circle [radius=.5];
\clip (0,0) circle [radius=.5];
\draw [thin, fill=white] (0,-.5) rectangle (.5,.5);
\draw[-<-=.5] (0,0)--(.5,0);
\node (p) at (0,0) [left]{${+}$};
\node at (0,0) [above right]{$v$};
\draw [fill=cyan] (0,0) circle [radius=1pt];
}\, if $v$ is a sink,
\item \tikz[baseline=-.6ex]{
\draw [thin, dashed, fill=white] (0,0) circle [radius=.5];
\draw[->-=.5] (0:0) -- (90:.5); 
\draw[->-=.5] (0:0) -- (210:.5); 
\draw[->-=.5] (0:0) -- (-30:.5);
\node (v) at (0,0) [above right]{$v$};
\fill (0,0) circle [radius=1pt];
}\,
or \,\tikz[baseline=-.6ex]{
\draw [thin, dashed] (0,0) circle [radius=.5];
\clip (0,0) circle [radius=.5];
\draw [thin, fill=white] (0,-.5) rectangle (.5,.5);
\draw[->-=.5] (0,0)--(.5,0);
\node (p) at (0,0) [left]{${-}$};
\node at (0,0) [above right]{$v$};
\draw [fill=cyan] (0,0) circle [radius=1pt];
}\, if $v$ is a source.
\end{itemize}
An {\em $A_2$ basis web} is the boundary-fixing isotopy class of a bipartite trivalent graph $G$ in $D_{\epsilon}$, 
where any internal face of $D\setminus G$ has at least six sides. 
We denote the set of $A_2$ basis webs in $D_{\epsilon}$ by $B(\epsilon_0,\dots,\epsilon_{m-1})$ or $B_\epsilon$.
For example,
$B{(+,-,+,-,+,-)}$ has the following $A_2$ basis webs:
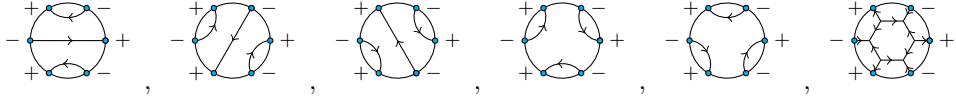
\[
\,\begin{tikzpicture}
\draw [thin, fill=white] (0,0) circle [radius=.5];
\draw[-<-=.5] (0:.5) -- (180:.5);
\draw[->-=.5] (60:.5) to[out=-120, in=-60] (120:.5);
\draw[-<-=.5] (240:.5) to[out=60, in=120] (300:.5);
\foreach \i in {0,1,...,6} \draw[fill=cyan] ($(0,0) !1! \i*60:(.5,0)$) circle [radius=1pt];
\node at (0:.5) [right]{$\scriptstyle{+}$};
\node at (60:.5) [right]{$\scriptstyle{-}$};
\node at (120:.5) [left]{$\scriptstyle{+}$};
\node at (180:.5) [left]{$\scriptstyle{-}$};
\node at (240:.5) [left]{$\scriptstyle{+}$};
\node at (300:.5) [right]{$\scriptstyle{-}$};
\end{tikzpicture}\,,
\,\begin{tikzpicture}
\begin{scope}[rotate=60]
\draw [thin, fill=white] (0,0) circle [radius=.5];
\draw[->-=.5] (0:.5) -- (180:.5);
\draw[-<-=.5] (60:.5) to[out=-120, in=-60] (120:.5);
\draw[->-=.5] (240:.5) to[out=60, in=120] (300:.5);
\foreach \i in {0,1,...,6} \draw[fill=cyan] ($(0,0) !1! \i*60:(.5,0)$) circle [radius=1pt];
\end{scope}
\node at (0:.5) [right]{$\scriptstyle{+}$};
\node at (60:.5) [right]{$\scriptstyle{-}$};
\node at (120:.5) [left]{$\scriptstyle{+}$};
\node at (180:.5) [left]{$\scriptstyle{-}$};
\node at (240:.5) [left]{$\scriptstyle{+}$};
\node at (300:.5) [right]{$\scriptstyle{-}$};
\end{tikzpicture}\,,
\,\begin{tikzpicture}
\begin{scope}[rotate=-60]
\draw [thin, fill=white] (0,0) circle [radius=.5];
\draw[->-=.5] (0:.5) -- (180:.5);
\draw[-<-=.5] (60:.5) to[out=-120, in=-60] (120:.5);
\draw[->-=.5] (240:.5) to[out=60, in=120] (300:.5);
\foreach \i in {0,1,...,6} \draw[fill=cyan] ($(0,0) !1! \i*60:(.5,0)$) circle [radius=1pt];
\end{scope}
\node at (0:.5) [right]{$\scriptstyle{+}$};
\node at (60:.5) [right]{$\scriptstyle{-}$};
\node at (120:.5) [left]{$\scriptstyle{+}$};
\node at (180:.5) [left]{$\scriptstyle{-}$};
\node at (240:.5) [left]{$\scriptstyle{+}$};
\node at (300:.5) [right]{$\scriptstyle{-}$};
\end{tikzpicture}\,,
\,\begin{tikzpicture}
\draw [thin, fill=white] (0,0) circle [radius=.5];
\draw[-<-=.5] (0:.5) to[out=180, in=-120] (60:.5);
\draw[rotate=120, -<-=.5] (0:.5) to[out=180, in=-120] (60:.5);
\draw[rotate=240, -<-=.5] (0:.5) to[out=180, in=-120] (60:.5);
\foreach \i in {0,1,...,6} \draw[fill=cyan] ($(0,0) !1! \i*60:(.5,0)$) circle [radius=1pt];
\node at (0:.5) [right]{$\scriptstyle{+}$};
\node at (60:.5) [right]{$\scriptstyle{-}$};
\node at (120:.5) [left]{$\scriptstyle{+}$};
\node at (180:.5) [left]{$\scriptstyle{-}$};
\node at (240:.5) [left]{$\scriptstyle{+}$};
\node at (300:.5) [right]{$\scriptstyle{-}$};
\end{tikzpicture}\,,
\,\begin{tikzpicture}
\begin{scope}[rotate=60]
\draw [thin, fill=white] (0,0) circle [radius=.5];
\draw[->-=.5] (0:.5) to[out=180, in=-120] (60:.5);
\draw[rotate=120, ->-=.5] (0:.5) to[out=180, in=-120] (60:.5);
\draw[rotate=240, ->-=.5] (0:.5) to[out=180, in=-120] (60:.5);
\foreach \i in {0,1,...,6} \draw[fill=cyan] ($(0,0) !1! \i*60:(.5,0)$) circle [radius=1pt];
\end{scope}
\node at (0:.5) [right]{$\scriptstyle{+}$};
\node at (60:.5) [right]{$\scriptstyle{-}$};
\node at (120:.5) [left]{$\scriptstyle{+}$};
\node at (180:.5) [left]{$\scriptstyle{-}$};
\node at (240:.5) [left]{$\scriptstyle{+}$};
\node at (300:.5) [right]{$\scriptstyle{-}$};
\end{tikzpicture}\,,
\,\begin{tikzpicture}
\draw [thin, fill=white] (0,0) circle [radius=.5];
\draw[-<-=.5] (0:.5) -- (0:.3);
\draw[rotate=60, ->-=.5] (0:.5) -- (0:.3);
\draw[rotate=120, -<-=.5] (0:.5) -- (0:.3);
\draw[rotate=180, ->-=.5] (0:.5) -- (0:.3);
\draw[rotate=240, -<-=.5] (0:.5) -- (0:.3);
\draw[rotate=300, ->-=.5] (0:.5) -- (0:.3);
\draw[->-=.5] (0:.3) -- (60:.3);
\draw[rotate=60, -<-=.5] (0:.3) -- (60:.3);
\draw[rotate=120, ->-=.5] (0:.3) -- (60:.3);
\draw[rotate=180, -<-=.5] (0:.3) -- (60:.3);
\draw[rotate=240, ->-=.5] (0:.3) -- (60:.3);
\draw[rotate=300, -<-=.5] (0:.3) -- (60:.3);
\foreach \i in {0,1,...,6} \draw[fill=cyan] ($(0,0) !1! \i*60:(.5,0)$) circle [radius=1pt];
\node at (0:.5) [right]{$\scriptstyle{+}$};
\node at (60:.5) [right]{$\scriptstyle{-}$};
\node at (120:.5) [left]{$\scriptstyle{+}$};
\node at (180:.5) [left]{$\scriptstyle{-}$};
\node at (240:.5) [left]{$\scriptstyle{+}$};
\node at (300:.5) [right]{$\scriptstyle{-}$};
\end{tikzpicture}\,.
\]
The {\em $A_2$ web space $W_\epsilon$} is the $\mathbb{Q}(q^{\frac{1}{6}})$-vector space spanned by $B_\epsilon$. 
A {\em tangled trivalent graph diagram} in $D_\epsilon$ is an immersed bipartite uni-trivalent graph in $D_\epsilon$ whose intersection points are only transverse double points of edges with crossing data 
\,\tikz[baseline=-.6ex, scale=.8]{
\draw [thin, dashed, fill=white] (0,0) circle [radius=.5];
\draw[->-=.8] (-45:.5) -- (135:.5);
\draw[->-=.8, white, double=black, double distance=0.4pt, ultra thick] (-135:.5) -- (45:.5);
}\, or 
\,\tikz[baseline=-.6ex, scale=.8]{
\draw [thin, dashed, fill=white] (0,0) circle [radius=.5];
\draw[->-=.8] (-135:.5) -- (45:.5);
\draw[->-=.8, white, double=black, double distance=0.4pt, ultra thick] (-45:.5) -- (135:.5);
}\, .
Tangled trivalent graph diagrams $G$ and $G'$ are regularly isotopic if $G$ is obtained from $G'$ by a finite sequence of boundary-fixing isotopies and (R1'), (R2), (R3) and (R4) moves with some direction of edges (see Fig.~\ref{Reidemeister}).
\begin{figure}
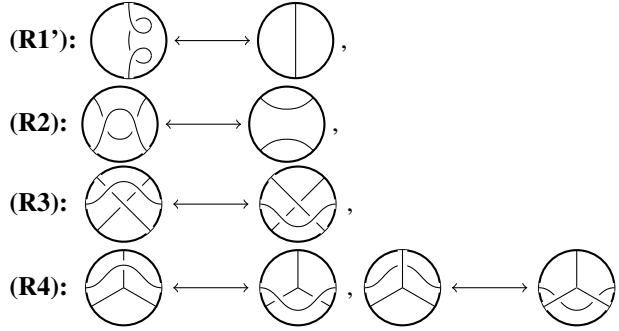

\begin{description}
\item[(R1')]
\tikz[baseline=-.6ex]{
\draw [thick] (0,0) circle [radius=.5];
\draw (.3,-.2)
to[out=south, in=east] (.2,-.3)
to[out=west, in=south] (.0,.0)
to[out=north, in=west] (.2,.3)
to[out=east, in=north] (.3,.2);
\draw[white, double=black, double distance=0.4pt, ultra thick] (0,-.5) 
to[out=north, in=west] (.2,-.1)
to[out=east, in=north] (.3,-.2);
\draw[white, double=black, double distance=0.4pt, ultra thick] (.3,.2)
to[out=south, in=east] (.2,.1)
to[out=west, in=south] (0,.5);
}
\tikz[baseline=-.6ex]{
\draw [<->, xshift=1.5cm] (1,0)--(2,0);
}
\tikz[baseline=-.6ex]{
\draw[xshift=3cm, thick] (0,0) circle [radius=.5];
\draw[xshift=3cm] (90:.5) to (-90:.5);
}\ ,
\item[(R2)]
\tikz[baseline=-.6ex]{
\draw [thick] (0,0) circle [radius=.5];
\draw (135:.5) to [out=south east, in=west](0,-.2) to [out=east, in=south west](45:.5);
\draw [white, double=black, double distance=0.4pt, ultra thick](-135:.5) to [out=north east, in=left](0,.2) to [out=right, in=north west] (-45:.5);
}
\tikz[baseline=-.6ex]{
\draw [<->, xshift=1.5cm] (1,0)--(2,0);
}
\tikz[baseline=-.6ex]{
\draw[xshift=3cm, thick] (0,0) circle [radius=.5];
\draw[xshift=3cm] (135:.5) to [out=south east, in=west](0,.2) to [out=east, in=south west](45:.5);
\draw[xshift=3cm] (-135:.5) to [out=north east, in=left](0,-.2) to [out=right, in=north west] (-45:.5);
}\ ,
\item[(R3)]
\tikz[baseline=-.6ex]{
\draw[thick] (0,0) circle [radius=.5];
\draw (-135:.5) -- (45:.5);
\draw[white, double=black, double distance=0.4pt, ultra thick] (135:.5) -- (-45:.5);
\draw[white, double=black, double distance=0.4pt, ultra thick](180:.5) to [out=right, in=left](0,.3) to [out=right, in=left] (-0:.5);
}
\tikz[baseline=-.6ex]{
\draw[<->, xshift=1.5cm] (1,0)--(2,0);
}
\tikz[baseline=-.6ex]{
\draw [xshift=3cm, thick] (0,0) circle [radius=.5];
\draw [xshift=3cm] (-135:.5) -- (45:.5);
\draw [xshift=3cm, white, double=black, double distance=0.4pt, ultra thick] (135:.5) -- (-45:.5);
\draw[xshift=3cm, white, double=black, double distance=0.4pt, ultra thick](180:.5) to [out=right, in=left](0,-.3) to [out=right, in=left] (-0:.5);
}\ ,
\item[(R4)]
\tikz[baseline=-.6ex]{
\draw[thick] (0,0) circle [radius=.5];
\draw (0:0) -- (90:.5); 
\draw (0:0) -- (210:.5); 
\draw (0:0) -- (-30:.5);
\draw[white, double=black, double distance=0.4pt, ultra thick](180:.5) to [out=right, in=left](0,.3) to [out=right, in=left] (-0:.5);
}
\tikz[baseline=-.6ex]{
\draw[<->, xshift=1.5cm] (1,0)--(2,0);
}
\tikz[baseline=-.6ex]{
\draw[thick] (0,0) circle [radius=.5];
\draw (0:0) -- (90:.5); 
\draw (0:0) -- (210:.5); 
\draw (0:0) -- (-30:.5);
\draw[white, double=black, double distance=0.4pt, ultra thick](180:.5) to [out=right, in=left](0,-.3) to [out=right, in=left] (-0:.5);
}\ , 
\tikz[baseline=-.6ex]{
\draw[thick] (0,0) circle [radius=.5];
\draw[white, double=black, double distance=0.4pt, ultra thick](180:.5) to [out=right, in=left](0,.3) to [out=right, in=left] (-0:.5);
\draw[white, double=black, double distance=0.4pt, ultra thick] (0:0) -- (90:.5); 
\draw (0:0) -- (210:.5); 
\draw (0:0) -- (-30:.5);
}
\tikz[baseline=-.6ex]{
\draw[<->, xshift=1.5cm] (1,0)--(2,0);
}
\tikz[baseline=-.6ex]{
\draw[thick] (0,0) circle [radius=.5];
\draw[white, double=black, double distance=0.4pt, ultra thick](180:.5) to [out=right, in=left](0,-.3) to [out=right, in=left] (-0:.5);
\draw (0:0) -- (90:.5); 
\draw[white, double=black, double distance=0.4pt, ultra thick] (0:0) -- (210:.5); 
\draw[white, double=black, double distance=0.4pt, ultra thick] (0:0) -- (-30:.5);
}.
\end{description}
\caption{The Reidemeister moves for tangled trivalent graph diagrams}
\label{Reidemeister}
\end{figure}
{\em Tangled trivalent graphs} in $D_\epsilon$ are regular isotopy classes of tangled trivalent graph diagrams in $D_\epsilon$. 
We denote $T_\epsilon$ the set of tangled trivalent graphs in $D_\epsilon$.
The diagram below is an example of a tangled trivalent graph diagram in $D{({+},{-},{+},{+},{+})}$. 
\begin{center}
\begin{tikzpicture}
\draw [thin, fill=white] (0,0) circle [radius=1];
\draw[-<-=.5] (0:1) -- (.7,0);
\draw[->-=.8] (.7,0) to[out=west, in=south east](144:1);
\draw[->-=.8] [white, double=black, double distance=0.4pt, ultra thick]
(72:1) -- (216:1);
\draw[->, rotate=-90] (.3,.4) to[out=west, in=north] (.1,.1) to[out=south, in=west] (.3,-.2);
\draw[white, double=black, double distance=0.4pt, ultra thick] 
(288:1) to[out=north west, in=west] (.4,.3);
\draw[-<-=.2] (.4,.3) to[out=east, in=north] (.7,0);
\draw[rotate=-90, white, double=black, double distance=0.4pt, ultra thick] (.3,.4) to[out=east, in=north] (.5,.1) to[out=south, in=east] (.3,-.2);
\foreach \i in {0,1,...,5} 
\draw[fill=cyan] ($(0,0) !1! \i*72:(1,0)$) circle [radius=1pt];
\node at (0:1) [right]{$\scriptstyle{+}$}
node at (72:1) [above]{$\scriptstyle{-}$} 
node at (144:1) [above left]{$\scriptstyle{+}$} 
node at (216:1) [below left]{$\scriptstyle{+}$} 
node at (288:1) [below]{$\scriptstyle{+}$};
\end{tikzpicture}
\end{center}
\begin{DEF}[The $A_2$ bracket~\cite{Kuperberg96}]
We define a $\mathbb{Q}(q^{\frac{1}{6}})$-linear map $\langle\,\cdot\,\rangle_3\colon\mathbb{Q}(q^{\frac{1}{6}})T_\epsilon\to W_\epsilon$ by the following.
\begin{itemize}
\item 
$\Big\langle\,\tikz[baseline=-.6ex]{
\draw [thin, dashed, fill=white] (0,0) circle [radius=.5];
\draw[->-=.8] (-45:.5) -- (135:.5);
\draw[->-=.8, white, double=black, double distance=0.4pt, ultra thick] (-135:.5) -- (45:.5);
}\,\Big\rangle_{\! 3}
=
q^{\frac{1}{3}}\Big\langle\,\tikz[baseline=-.6ex]{
\draw[thin, dashed, fill=white] (0,0) circle [radius=.5];
\draw[->-=.5] (-45:.5) to [out=north west, in=south](.2,0) to [out=north, in=south west](45:.5);
\draw[->-=.5] (-135:.5) to [out=north east, in=south](-.2,0) to [out=north, in=south east] (135:.5);
}\,\Big\rangle_{\! 3}
-
q^{-\frac{1}{6}}\Big\langle\,\tikz[baseline=-.6ex]{
\draw[thin, dashed, fill=white] (0,0) circle [radius=.5];
\draw[->-=.5] (-45:.5) -- (0,-.2);
\draw[->-=.5] (-135:.5) -- (0,-.2);
\draw[-<-=.5] (0,-.2) -- (0,.2);
\draw[-<-=.5] (45:.5) -- (0,.2);
\draw[-<-=.5] (135:.5) -- (0,.2);
}\,\Big\rangle_{\! 3}
$,
\item 
$\Big\langle\,\tikz[baseline=-.6ex]{
\draw [thin, dashed, fill=white] (0,0) circle [radius=.5];
\draw[->-=.8] (-135:.5) -- (45:.5);
\draw[->-=.8, white, double=black, double distance=0.4pt, ultra thick] (-45:.5) -- (135:.5);
}\,\Big\rangle_{\! 3}
=
q^{-\frac{1}{3}}\Big\langle\,\tikz[baseline=-.6ex]{
\draw[thin, dashed, fill=white] (0,0) circle [radius=.5];
\draw[->-=.5] (-45:.5) to [out=north west, in=south](.2,0) to [out=north, in=south west](45:.5);
\draw[->-=.5] (-135:.5) to [out=north east, in=south](-.2,0) to [out=north, in=south east] (135:.5);
}\,\Big\rangle_{\! 3}
-
q^{\frac{1}{6}}\Big\langle\,\tikz[baseline=-.6ex]{
\draw[thin, dashed, fill=white] (0,0) circle [radius=.5];
\draw[->-=.5] (-45:.5) -- (0,-.2);
\draw[->-=.5] (-135:.5) -- (0,-.2);
\draw[-<-=.5] (0,-.2) -- (0,.2);
\draw[-<-=.5] (45:.5) -- (0,.2);
\draw[-<-=.5] (135:.5) -- (0,.2);
}\,\Big\rangle_{\! 3}
$,
\item 
$\Big\langle\,\tikz[baseline=-.6ex, rotate=90]{
\draw[thin, dashed, fill=white] (0,0) circle [radius=.5];
\draw[->-=.6] (-45:.5) -- (-45:.3);
\draw[->-=.6] (-135:.5) -- (-135:.3);
\draw[-<-=.6] (45:.5) -- (45:.3);
\draw[->-=.6] (135:.5) -- (135:.3);
\draw[->-=.5] (45:.3) -- (135:.3);
\draw[-<-=.5] (-45:.3) -- (-135:.3);
\draw[->-=.5] (45:.3) -- (-45:.3);
\draw[-<-=.5] (135:.3) -- (-135:.3);
}\,\Big\rangle_{\! 3}
=
\Big\langle\,\tikz[baseline=-.6ex, rotate=90]{
\draw[thin, dashed, fill=white] (0,0) circle [radius=.5];
\draw[->-=.5] (-45:.5) to [out=north west, in=south](.2,0) to [out=north, in=south west](45:.5);
\draw[-<-=.5] (-135:.5) to [out=north east, in=south](-.2,0) to [out=north, in=south east] (135:.5);
}\,\Big\rangle_{\! 3}
+
\Big\langle\,\tikz[rotate=90, baseline=-.6ex, rotate=90]{
\draw[thin, dashed, fill=white] (0,0) circle [radius=.5];
\draw[-<-=.5] (-45:.5) to [out=north west, in=south](.2,0) to [out=north, in=south west](45:.5);
\draw[->-=.5] (-135:.5) to [out=north east, in=south](-.2,0) to [out=north, in=south east] (135:.5);
}\,\Big\rangle_{\! 3}
$,
\item 
$\Big\langle\,\tikz[baseline=-.6ex, rotate=-90]{
\draw[thin, dashed, fill=white] (0,0) circle [radius=.5];
\draw[->-=.5] (0,-.5) -- (0,-.25);
\draw[->-=.5] (0,.25) -- (0,.5);
\draw[-<-=.5] (0,-.25) to [out=60, in=120, relative](0,.25);
\draw[-<-=.5] (0,-.25) to [out=-60, in=-120, relative](0,.25);
}\,\Big\rangle_{\! 3}
=
\left[2\right]\Big\langle\,\tikz[baseline=-.6ex, rotate=-90]{
\draw[thin, dashed, fill=white] (0,0) circle [radius=.5];
\draw[->-=.5] (0,-.5) -- (0,.5);
}\,\Big\rangle_{\! 3}
$,
\item 
$
\Big\langle G\sqcup
\,\tikz[baseline=-.6ex]{
\draw[thin, dashed, fill=white] (0,0) circle [radius=.5];
\draw[->-=.5] (0,0) circle [radius=.3];
}\,\Big\rangle_{\! 3}
=
\left[3\right] \langle G\rangle_{3}
$.
\end{itemize}
We can confirm that this map is well-defined, that is, the map is invariant under the Reidemeister moves.
\end{DEF}

We next consider a $A_2$ web space $W{(n_1^{\epsilon_1},n_2^{\epsilon_2},\dots,n_l^{\epsilon_l})}$ whose marked point $p$ such that $p_{n_1+n_2+\dots+n_j+1} \leq p \leq p_{n_1+n_2+\dots+n_{j+1}}$ is decorated with $\epsilon_{j+1}$ for $j=0,1,\dots, l-1$. 
The index of $p_{n_1+n_2+\dots+n_j+1}$ is $p_1$ if $j=0$.
We define the $A_2$ clasp
$\Big\langle\tikz[baseline=-.6ex]{
\draw[->-=.8] (-.5,0) -- (.5,0);
\draw[fill=white] (-.1,-.3) rectangle (.1,.3);
\node at (.1,0) [above right]{${\scriptstyle n}$};
}\,\Big\rangle_3\in W_{(n^{+},n^{-})}
$ 
and introduce well-known properties of it. (See, for example, \ref{Kuperberg96} and \ref{OhtsukiYamada97}.) 
\begin{DEF}(The $A_2$ clasp of type $n$)
the $A_2$ clasp 
$\Big\langle
%%%%%%%%%%
\tikz[baseline=-.6ex]{
\draw[->-=.8] (-.5,0) -- (.5,0);
\draw[fill=white] (-.1,-.3) rectangle (.1,.3);
\node at (.1,0) [above right]{${\scriptstyle n}$};}
%%%%%%%%%%
\,\Big\rangle$ 
is defined inductively as follows. 
For $n=1$,
\[\Big\langle
%%%%%%%%%%
\tikz[baseline=-.6ex]{
\draw[->-=.8] (-.5,0) -- (.5,0);
\draw[fill=white] (-.1,-.3) rectangle (.1,.3);
\node at (.1,0) [above right]{${\scriptstyle 1}$};}
%%%%%%%%%%
\,\Big\rangle_3
=
\,
%%%%%%%%%%
\tikz[baseline=-.6ex]{
\draw[->-=.5] (-.5,0) -- (.5,0) node at (0,0) [above]{${\scriptstyle 1}$};}
%%%%%%%%%%
\,\in W_{(1^{+},1^{-})}\]
For any $n>1$,
\[
\Big\langle
%%%%%%%%%%
\tikz[baseline=-.6ex]{
\draw[->-=.8] (-.5,0) -- (.5,0);
\draw[fill=white] (-.1,-.3) rectangle (.1,.3);
\node at (.1,0) [above right]{${\scriptstyle n}$};}
%%%%%%%%%%
\,\Big\rangle_3
=
\bigg\langle\,
%%%%%%%%%%
\tikz[baseline=-.6ex]{
\draw[->-=.8] (-.5,.1) -- (.5,.1);
\draw[->-=.5] (-.5,-.4) -- (.5,-.4);
\draw[fill=white] (-.1,-.2) rectangle (.1,.4);
\node at (0,0) [above right]{${\scriptstyle n-1}$};
\node at (.2,-.4) [above right]{${\scriptstyle 1}$};}
%%%%%%%%%%
\,\bigg\rangle_{\! 3}
-\frac{\left[n-1\right]}{\left[n\right]}
\bigg\langle\,
%%%%%%%%%%
\tikz[baseline=-.6ex]{
\draw[->-=.5] (-.9,.1) -- (-.4,.1);
\draw[->-=.5] (-.3,.2) -- (.3,.2);
\draw[->-=.5] (.4,.1) -- (.9,.1);
\draw[->-=.5] (-.9,-.4) 
to[out=east, in=west] (-.3,-.4) 
to[out=east, in=south] (-.1,-.2);
\draw[-<-=.8] (-.1,-.2) 
to[out=north, in=east] (-.3,0);
\draw[-<-=.5] (.9,-.4) 
to[out=west, in=east] (.3,-.4)
to[out=west, in=south] (.1,-.2);
\draw[->-=.8] (.1,-.2)
to[out=north, in=west] (.3,0);
\draw[fill=white] (-.4,-.2) rectangle (-.3,.4);
\draw[fill=white] (.3,-.2) rectangle (.4,.4);
\draw[-<-=.5] (-.1,-.2) -- (.1,-.2);
\node at (-.3,0) [above left]{${\scriptscriptstyle n-1}$};
\node at (.3,0) [above right]{${\scriptscriptstyle n-1}$};
\node at (0,.1) [above]{${\scriptscriptstyle n-2}$};
\node at (-.6,-.3) {${\scriptscriptstyle 1}$};
\node at (.6,-.3) {${\scriptscriptstyle 1}$};
\node at (0,-.2) [below]{${\scriptscriptstyle 1}$};
\node at (-.3,0) [right]{${\scriptscriptstyle 1}$};
\node at (.3,0) [left]{${\scriptscriptstyle 1}$};}
%%%%%%%%%%
\,\bigg\rangle_{\! 3} \in W(n^{+},n^{-})\]
\end{DEF}

The $A_2$ clasps have the following properties.
\begin{LEM}[Properties of $A_2$ clasps]
For any positive integer $n$,
\begin{itemize} 
\item $\Big\langle\,\tikz[baseline=-.6ex]{
\draw[->-=.5] (-.6,0) -- (.6,0);
\draw[fill=white] (-.4,-.3) rectangle (-.2,.3);
\draw[fill=white] (.2,-.3) rectangle (.4,.3);
\node at (-.3,0) [above right]{${\scriptstyle n}$};
}\,\Big\rangle_{\! 3}
=
\Big\langle\,\tikz[baseline=-.6ex]{
\draw[->-=.8] (-.5,0) -- (.5,0);
\draw[fill=white] (-.1,-.3) rectangle (.1,.3);
\node at (0,0) [above right] {${\scriptstyle n}$};
}\,\Big\rangle_3
$,
\item $\Big\langle\,\tikz[baseline=-.6ex]{
\draw[->-=.5] (-.5,0) -- (-.1,0);
\draw[->-=.5] (0,.2) -- (.5,.2);
\draw[->-=.5] (0,-.2) -- (.5,-.2);
\draw[->-=.5] (0,.1) 
to[out=east, in=north] (.3,0);
\draw[->-=.5] (0,-.1)
to[out=east, in=south] (.3,0);
\draw[-<-=.5] (.3,0) -- (.5,0);
\draw[fill=white] (-.2,-.3) rectangle (0,.3);
\node at (.4,0) [right] {${\scriptstyle 1}$};
\node at (.4,.2) [right] {${\scriptstyle n-k-2}$};
\node at (.4,-.2) [right] {${\scriptstyle k}$};
}\,\Big\rangle_{\! 3}=0$\quad ($k=0,1,\dots,n-2$).
\end{itemize}
\end{LEM}

We also define the $A_2$ clasp in $W{(m^{+},n^{-},n^{+},m^{-})}$ according to Ohtsuki and Yamada~\cite{OhtsukiYamada97}.
\begin{DEF}[the $A_2$ clasp in $(m,n)$]\label{doubleA2clasp}
\[\Big\langle\,
\tikz[baseline=-.6ex, scale=.8]{
\draw[-<-=.8] (-.6,.4) -- (.6,.4);
\draw[->-=.8] (-.6,-.4) -- (.6,-.4);
\draw[fill=white] (-.1,-.6) rectangle (.1,.6);
\draw (-.1,.0) -- (.1,.0);
\node at (.4,-.6)[right]{$\scriptstyle{m}$};
\node at (-.4,-.6)[left]{$\scriptstyle{m}$};
\node at (.4,.6)[right]{$\scriptstyle{n}$};
\node at (-.4,.6)[left]{$\scriptstyle{n}$};
}\,\Big\rangle_3
=
\sum_{k=0}^{\min\{m,n\}}
(-1)^k
\frac{{n\brack k}{m\brack k}}{{n+m+1\brack k}}
\Big\langle\,\tikz[baseline=-.6ex]{
\draw
(-.4,.4) -- +(-.2,0)
(.4,-.4) -- +(.2,0)
(-.4,-.4) -- +(-.2,0)
(.4,.4) -- +(.2,0);
\draw[-<-=.5] (-.4,.5) -- (.4,.5);
\draw[->-=.5] (-.4,-.5) -- (.4,-.5);
\draw[-<-=.5] (-.4,.3) to[out=east, in=east] (-.4,-.3);
\draw[->-=.5] (.4,.3) to[out=west, in=west] (.4,-.3);
\draw[fill=white] (.4,-.6) rectangle +(.1,.4);
\draw[fill=white] (-.4,-.6) rectangle +(-.1,.4);
\draw[fill=white] (.4,.6) rectangle +(.1,-.4);
\draw[fill=white] (-.4,.6) rectangle +(-.1,-.4);
\node at (.4,-.6)[right]{$\scriptstyle{m}$};
\node at (-.4,-.6)[left]{$\scriptstyle{m}$};
\node at (.4,.6)[right]{$\scriptstyle{n}$};
\node at (-.4,.6)[left]{$\scriptstyle{n}$};
\node at (0,.5)[above]{$\scriptstyle{n-k}$};
\node at (0,-.5)[below]{$\scriptstyle{m-k}$};
\node at (-.2,0)[left]{$\scriptstyle{k}$};
\node at (.2,0)[right]{$\scriptstyle{k}$};
}\Big\rangle_3\,\in W(m^{+},n^{-},n^{-},m^{+}).
\]
\end{DEF}
We remark that the $A_2$ clasp in $W(m^{-},n^{+},n^{-},m^{+})$ is also defined by the above recursive formula with opposite directed webs.

\begin{LEM}[Property of $A_2$ clasps of type $(m,n)$]\label{doubleA2claspprop}
\[
\big\langle\,\tikz[baseline=-.6ex]{
\draw[-<-=.5] (-.5,.2) -- (-.1,.2);
\draw[->-=.5] (-.5,-.2) -- (-.1,-.2);
\draw[-<-=.5] (0,.2) -- (.5,.2);
\draw[->-=.5] (0,-.2) -- (.5,-.2);
\draw[-<=.5] (0,.1) 
to[out=east, in=north] (.3,0);
\draw (0,-.1)
to[out=east, in=south] (.3,0);
\draw[fill=white] (-.2,-.3) rectangle (0,.3);
\draw (-.2,.0) -- (.0,.0);
\node at (.4,0) [right] {${\scriptstyle 1}$};
\node at (.4,.2) [right] {${\scriptstyle n-1}$};
\node at (.4,-.2) [right] {${\scriptstyle m-1}$};
}\,\big\rangle_3=0 .
\]
\end{LEM}

The clasped $A_2$ web space $W((m_1^{\epsilon_1}, n_1^{\epsilon'_1}),(m_2^{\epsilon_2}, n_2^{\epsilon'_2}),\dots,(m_k^{\epsilon_k}, n_k^{\epsilon'_k}))$ is a subspace of $W(m_1^{\epsilon_1}, n_1^{\epsilon'_1},\dots,m_k^{\epsilon_k}, n_k^{\epsilon'_k})$. 
An $A_2$ web in the subspace is clasped by the $A_2$ clasp of type $(m_j^{\epsilon_j},n_j^{\epsilon'_j})$ at near the marked points decorated by $m_j^{\epsilon_j}$ and $n_j^{\epsilon'_j}$ where $\epsilon\neq\epsilon'$ and $j=1,2,\dots,k$.

\begin{DEF}[A clasped $A_2$ web space]
Let $m_j, n'_j$ be non-negative integers and $\epsilon_j\neq\epsilon'_j$ signs for $j=1,2,\dots,k$. 
We define a subspace $W((m_1^{\epsilon_1}, n_1^{\epsilon'_1}),(m_2^{\epsilon_2}, n_2^{\epsilon'_2}),\dots,(m_k^{\epsilon_k}, n_k^{\epsilon'_k}))$ of $W(m_1^{\epsilon_1}, n_1^{\epsilon'_1},\dots,m_k^{\epsilon_k}, n_k^{\epsilon'_k})$ called the {\em clasped $A_2$ web space} as follows:
\begin{align*}
&W((m_1^{\epsilon_1}, n_1^{\epsilon'_1}),(m_2^{\epsilon_2}, n_2^{\epsilon'_2}),\dots,(m_k^{\epsilon_k}, n_k^{\epsilon'_k}))\\
&\quad=\Bigg\{\,
\bigg\langle\tikz[baseline=-.6ex]{
\draw [thin, fill=white] (0,0) circle [radius=1];
\draw[triple={[line width=1.4pt, white] in [line width=2.2pt, black] in [line width=5.4pt, white]}] (0,0) -- (-30:1);
\draw[triple={[line width=1.4pt, white] in [line width=2.2pt, black] in [line width=5.4pt, white]}] (0,0) -- (0:1);
\draw[triple={[line width=1.4pt, white] in [line width=2.2pt, black] in [line width=5.4pt, white]}] (0,0) -- (30:1);
\draw[triple={[line width=1.4pt, white] in [line width=2.2pt, black] in [line width=5.4pt, white]}] (0,0) -- (60:1);
\coordinate (a) at ($(0,0)!.8!-30:(-12:1)$);
\coordinate (b) at ($(0,0)!.9!-30:(12:1)$);
\coordinate (a0) at ($(0,0)!.8!0:(-12:1)$);
\coordinate (b0) at ($(0,0)!.9!0:(12:1)$);
\coordinate (a1) at ($(0,0)!.8!30:(-12:1)$);
\coordinate (b1) at ($(0,0)!.9!30:(12:1)$);
\coordinate (a2) at ($(0,0)!.8!60:(-12:1)$);
\coordinate (b2) at ($(0,0)!.9!60:(12:1)$);
\foreach \i in {0,1,...,11} \draw[fill=cyan] ($(0,0) !1! \i*30:(1,0)$) circle [radius=1pt];
\draw[fill=cyan] (-30:1) circle [radius=1pt] node [below right]{${\scriptscriptstyle (m_k^{\epsilon_k},n_k^{\epsilon'_k})}$};
\draw[fill=cyan] (0:1) circle [radius=1pt] node [right]{${\scriptscriptstyle (m_1^{\epsilon_1},n_k^{\epsilon'_1})}$};
\draw[fill=cyan] (30:1) circle [radius=1pt] node [above right]{${\scriptscriptstyle (m_2^{\epsilon_2},n_k^{\epsilon'_2})}$};
\node[rotate=10] at (100:.8){$\cdots$} node[rotate=25] at (-65:.8){$\cdots$} node[rotate=-60] at (-150:.8){$\cdots$};
\node[circle, draw, fill=lightgray!30] (0,0) {\quad$w$\quad\quad};
}\bigg\rangle_{\!3}
\;{\Bigg\vert}\; w\in W(m_1^{\epsilon_1}, n_1^{\epsilon'_1},\dots,m_k^{\epsilon_k}, n_k^{\epsilon'_k})\,\Bigg\},
\end{align*}
where
%%%%%%%%%%
$
\tikz[baseline=-.6ex, scale=0.5]{
\draw[triple={[line width=1.4pt, white] in [line width=2.2pt, black] in [line width=5.4pt, white]}]
(-1,0)--(1,0);
\draw[fill=cyan] (1,0) circle (.1);
\node at (1,0) [right]{$\scriptstyle{(m_j^{\epsilon_j},n_j^{\epsilon'_j})}$};
}
$
%%%%%%%%%%
denotes 
%%%%%%%%%%
$
\tikz[baseline=-.6ex, scale=.5]{
\draw[-<-=.8] (-.6,.4) -- (.6,.4);
\draw[->-=.8] (-.6,-.4) -- (.6,-.4);
\draw[fill=white] (-.1,-.6) rectangle (.1,.6);
\draw (-.1,.0) -- (.1,.0);
\node at (.4,-.6)[right]{$\scriptstyle{m_j}$};
\node at (.4,.6)[right]{$\scriptstyle{n_j}$};}
$
%%%%%%%%%%
if $\epsilon_j={+}$ and
%%%%%%%%%%
$
\tikz[baseline=-.6ex, scale=.5]{
\draw[->-=.8] (-.6,.4) -- (.6,.4);
\draw[-<-=.8] (-.6,-.4) -- (.6,-.4);
\draw[fill=white] (-.1,-.6) rectangle (.1,.6);
\draw (-.1,.0) -- (.1,.0);
\node at (.4,-.6)[right]{$\scriptstyle{m_j}$};
\node at (.4,.6)[right]{$\scriptstyle{n_j}$};}
$
%%%%%%%%%%
if $\epsilon_j={-}$.
\end{DEF}

In the following, 
we use trivalent graphs with white vertices to represent certain types of $A_2$ webs. 
This notation derives some important formulae and it is useful to calculate clasped $A_2$ web spaces. 

\begin{DEF}
Let $n$ be a non-negative integer. 
For $0\leq i\leq n$, 
we define two types of trivalent vertices as follows.
\[
%%%%%%%%%%
\tikz[baseline=-.6ex, scale=0.5]{
\draw[triple={[line width=1.4pt, white] in [line width=2.2pt, black] in [line width=5.4pt, white]}]
(0,0) -- (1,0);
\draw[-<-=.5] (-1,1) to[out=east, in=north west] (.0,.0);
\draw[->-=.5] (-1,-1) to[out=east, in=south west] (.0,.0);
\draw[fill=white] (0,0)  circle (.2);
\node at (-1,1) [above]{$\scriptstyle{n}$};
\node at (-1,-1) [above]{$\scriptstyle{n}$};
\node at (.5,0) [above]{$\scriptstyle{i}$};}
%%%%%%%%%%
\text{\ is defined by\ }
%%%%%%%%%%
\tikz[baseline=-.6ex, scale=0.5]{
\draw (-1,.9) -- +(-.5,0);
\draw (-1,-.9) -- +(-.5,0);
\draw (1,.3) -- +(.5,0);
\draw (1,-.3) -- +(.5,0);
\draw[-<-=.5] (-1,1) to[out=east, in=west] (1,.3);
\draw[->-=.5] (-1,-1) to[out=east, in=west] (1,-.3);
\draw[-<-=.5] (-1,.8) to[out=east, in=east] (-1,.-.8);
\draw[fill=white] (-1.2,.6) rectangle (-1,1.2);
\draw[fill=white] (-1.2,-.6) rectangle (-1,-1.2);
\draw[fill=white] (1,-.6) rectangle (1.2,.6);
\draw (1,.0) -- (1.2,.0);
\node at (-1,.6)[left]{$\scriptstyle{n}$};
\node at (-1,-.6)[left]{$\scriptstyle{n}$};
\node at (-1.2,0){$\scriptstyle{n-i}$};
\node at (0,1){$\scriptstyle{i}$};
\node at (0,-1){$\scriptstyle{i}$};}
%%%%%%%%%%
\, ,
%%%%%%%%%%
\tikz[baseline=-.6ex, scale=0.5]{
\draw[triple={[line width=1.4pt, white] in [line width=2.2pt, black] in [line width=5.4pt, white]}]
(0,0) -- (1,0);
\draw[->-=.5] (-1,1) to[out=east, in=north west] (.0,.0);
\draw[-<-=.5] (-1,-1) to[out=east, in=south west] (.0,.0);
\draw[fill=white] (0,0)  circle (.2);
\node at (-1,1) [above]{$\scriptstyle{n}$};
\node at (-1,-1) [above]{$\scriptstyle{n}$};
\node at (.5,0) [above]{$\scriptstyle{i}$};}
%%%%%%%%%%
\text{\ by\ }
%%%%%%%%%%
\tikz[baseline=-.6ex, scale=0.5]{
\draw (-1,.9) -- +(-.5,0);
\draw (-1,-.9) -- +(-.5,0);
\draw (1,.3) -- +(.5,0);
\draw (1,-.3) -- +(.5,0);
\draw[->-=.5] (-1,1) to[out=east, in=west] (1,.3);
\draw[-<-=.5] (-1,-1) to[out=east, in=west] (1,-.3);
\draw[->-=.5] (-1,.8) to[out=east, in=east] (-1,.-.8);
\draw[fill=white] (-1.2,.6) rectangle (-1,1.2);
\draw[fill=white] (-1.2,-.6) rectangle (-1,-1.2);
\draw[fill=white] (1,-.6) rectangle (1.2,.6);
\draw (1,.0) -- (1.2,.0);
\node at (-1,.6)[left]{$\scriptstyle{n}$};
\node at (-1,-.6)[left]{$\scriptstyle{n}$};
\node at (-1.2,0){$\scriptstyle{n-i}$};
\node at (0,1){$\scriptstyle{i}$};
\node at (0,-1){$\scriptstyle{i}$};}\,.
%%%%%%%%%%
\]
\end{DEF}

We use the following notations:
\begin{itemize}
\item $\Delta(m,n)=
\bigg\langle\,
\tikz[baseline=-.6ex, scale=0.5]{
\draw[-<-=.5] (0,0) circle (.6);
\draw[->-=.5] (0,0) circle (1);
\draw[fill=white] (.4,-.1) rectangle (1.2,.1);
\draw (.8,-.1) -- (.8,.1);
\node at (-.6,0) [right]{$\scriptstyle{n}$};
\node at (-1,0) [left]{$\scriptstyle{m}$};
}\,\bigg\rangle_{\! 3},
$
\item $\theta(n,n,(i,i))=
\bigg\langle\,
\tikz[baseline=-.6ex, scale=0.5]{
\draw[triple={[line width=1.4pt, white] in [line width=2.2pt, black] in [line width=5.4pt, white]}]
(0,0) -- (1,0);
\draw[->-=.5] 
(0,0) to[out=north, in=west] 
(.5,.8) to[out=east, in=north] (1,0);
\draw[-<-=.5] 
(0,0) to[out=south, in=west] 
(.5,-.8) to[out=east, in=south] (1,0);
\draw[fill=white] (0,0)  circle (.1);
\draw[fill=white] (1,0)  circle (.1);
\node at (0,.8) {$\scriptstyle{n}$};
\node at (0,-.8) {$\scriptstyle{n}$};
\node at (.5,.0) [above]{$\scriptstyle{i}$};
}\,\bigg\rangle_{\! 3},
$
\item $\operatorname{Tet}\!
\begin{bmatrix}
n^{-}&n^{+}&(j,j)\\
n^{-}&n^{+}&(i,i)
\end{bmatrix}=
\bigg\langle\,
\tikz[baseline=-.6ex, scale=0.5]{
\draw[triple={[line width=1.4pt, white] in [line width=2.2pt, black] in [line width=5.4pt, white]}]
(0,0) -- (1.2,0);
\draw[triple={[line width=1.4pt, white] in [line width=2.2pt, black] in [line width=5.4pt, white]}]
(.6,1) to[out=north, in=north] 
(2,1) -- (2,-1)
to[out=south, in=south] (.6,-1);
\draw[->-=.5] (0,0) -- (.6,1);
\draw[->-=.5] (.6,1) -- (1.2,0);
\draw[-<-=.5] (0,0) -- (.6,-1);
\draw[-<-=.5] (.6,-1) -- (1.2,0);
\draw[fill=white] (0,0)  circle (.1);
\draw[fill=white] (1.2,0)  circle (.1);
\draw[fill=white] (.6,1)  circle (.1);
\draw[fill=white] (.6,-1)  circle (.1);
\node at (0,.8) {$\scriptstyle{n}$};
\node at (0,-.8) {$\scriptstyle{n}$};
\node at (1,.8) {$\scriptstyle{n}$};
\node at (1,-.8) {$\scriptstyle{n}$};
\node at (.6,.0) [above]{$\scriptstyle{i}$};
\node at (2,.0) [left]{$\scriptstyle{j}$};
}\,\bigg\rangle_{\! 3},
$
\item 
$\displaystyle
\begin{Bmatrix}
n^{-}&n^{+}&(j,j)\\
n^{-}&n^{+}&(i,i)
\end{Bmatrix}
=
\frac{\operatorname{Tet}\!
\begin{bmatrix}
n^{-}&n^{+}&(j,j)\\
n^{-}&n^{+}&(i,i)
\end{bmatrix}\Delta(j,j)}{\theta(n,n,(j,j))^2},
$
\end{itemize}
where $m,n$ are any non-negative integers and $0\leq i,j\leq n$.

\begin{LEM}[\cite{Yuasa17b}]\label{value}\ 
\begin{enumerate}
\item $\Delta(i,j)=\frac{\left[i+1\right]\left[j+1\right]\left[i+j+2\right]}{\left[2\right]}$,
\item $\theta(n,n,(i,i))=\sum_{k=0}^i(-1)^k\frac{{i\brack k}^2}{{2i+1\brack k}}\frac{\Delta(n,0)^2}{\Delta(n-i+k,0)}=\frac{{n+i+2\brack 2i+2}}{{n\brack i}^2}\Delta(i,i)$,
\item $\operatorname{Tet}\!
\begin{bmatrix}
n^{-}&n^{+}&(j,j)\\
n^{-}&n^{+}&(i,i)
\end{bmatrix}
=\sum_{k=\max\{0, i+j-n\}}^i(-1)^k\frac{{i\brack k}^2{n-j\brack i-k}{n+j+2\brack i-k}}{{2i+1\brack k}{n\brack i-k}^2}\theta(n,n,(j,j))$.
\end{enumerate}
\end{LEM}

\begin{LEM}[\cite{Yuasa17b}]\label{bigon}
\[
\bigg\langle\,
\tikz[baseline=-.6ex, scale=0.5]{
\draw[triple={[line width=1.4pt, white] in [line width=2.2pt, black] in [line width=5.4pt, white]}]
(-1,0)--(-.2,0);
\draw[triple={[line width=1.4pt, white] in [line width=2.2pt, black] in [line width=5.4pt, white]}]
(1.2,0)--(2,0);
\draw[-<-=.5] 
(-.2,0) to[out=north, in=west] 
(.6,.5) to[out=east, in=north] (1.2,0);
\draw[->-=.5] 
(-.2,0) to[out=south, in=west] 
(.6,-.5) to[out=east, in=south] (1.2,0);
\draw[fill=white] (-.2,0) circle (.1);
\draw[fill=white] (1.2,0) circle (.1);
\node at (.6,.5) [above]{$\scriptstyle{n}$};
\node at (.6,-.5) [below]{$\scriptstyle{n}$};
\node at (-1,0) [above]{$\scriptstyle{i}$};
\node at (2,0) [above]{$\scriptstyle{j}$};
}\,\bigg\rangle_{\! 3}
=\delta_{ij}
\frac{\theta(n,n,(i,i))}{\Delta(i,i)}
\bigg\langle
\,\tikz[baseline=-.6ex, scale=0.5]{
\draw[triple={[line width=1.4pt, white] in [line width=2.2pt, black] in [line width=5.4pt, white]}]
(-1,0)--(1,0);
\node at (0,0) [above]{$\scriptstyle{i}$};
}\,\bigg\rangle_{\! 3},
\]
where $\delta_{ij}$ is the Kronecker delta function.
\end{LEM}
We can easily show the following by the same way as the previous lemma.
\begin{LEM}\label{bibigon}
\[
\bigg\langle\,
\tikz[baseline=-.6ex, scale=0.5]{
\draw[-<-=.5] (-1,0)--(-.2,0);
\draw[-<-=.5] (1.2,0)--(2,0);
\draw[triple={[line width=1.4pt, white] in [line width=2.2pt, black] in [line width=5.4pt, white]}] 
(-.2,0) to[out=north, in=west] 
(.6,.5) to[out=east, in=north] (1.2,0);
\draw[-<-=.5] 
(-.2,0) to[out=south, in=west] 
(.6,-.5) to[out=east, in=south] (1.2,0);
\draw[fill=white] (-.2,0) circle (.1);
\draw[fill=white] (1.2,0) circle (.1);
\node at (.6,.5) [above]{$\scriptstyle{i}$};
\node at (.6,-.5) [below]{$\scriptstyle{n}$};
\node at (-1,0) [above]{$\scriptstyle{n}$};
\node at (2,0) [above]{$\scriptstyle{n}$};
}\,\bigg\rangle_{\! 3}
=
\frac{\theta(n,n,(i,i))}{\Delta(n,0)}
\bigg\langle
\,\tikz[baseline=-.6ex, scale=0.5]{
\draw[-<-=.5](-1,0)--(1,0);
\node at (0,0) [above]{$\scriptstyle{n}$};
}\,\bigg\rangle_{\! 3}\, .
\]
\end{LEM}
\begin{THM}[Recoupling Theorem for $A_2$]\label{recoupling}
Let $n$ be non-negative integer and $0\leq i,s,t\leq n$.
\begin{itemize}
\item
$\bigg\langle\tikz[baseline=-.6ex, scale=0.4]{
\draw[triple={[line width=1.4pt, white] in [line width=2.2pt, black] in [line width=5.4pt, white]}]
(0,0) -- (1,0);
\draw[-<-=.5] (-1,1) -- (0,0);
\draw[->-=.5] (-1,-1) -- (0,0);
\draw[-<-=.5] (1,0) -- (2,1);
\draw[->-=.5] (1,0) -- (2,-1);
\draw[fill=white] (0,0) circle (.1);
\draw[fill=white] (1,0) circle (.1);
\node at (-1,1) [below]{$\scriptstyle{n}$};
\node at (-1,-1) [above]{$\scriptstyle{n}$};
\node at (2,1) [below]{$\scriptstyle{n}$};
\node at (2,-1) [above]{$\scriptstyle{n}$};
\node at (.5,0) [above]{$\scriptstyle{i}$};
}\,\bigg\rangle_{\! 3}
=\sum_{j=0}^{n}
\begin{Bmatrix}
n^{-}&n^{+}&(j,j)\\
n^{-}&n^{+}&(i,i)
\end{Bmatrix}
\bigg\langle\tikz[baseline=-.6ex, scale=0.4]{
\draw[triple={[line width=1.4pt, white] in [line width=2.2pt, black] in [line width=5.4pt, white]}]
(0,-.5) -- (0,.5);
\draw[-<-=.5] (-1,1) -- (0,.5);
\draw[->-=.5] (-1,-1) -- (0,-.5);
\draw[-<-=.5] (0,.5) -- (1,1);
\draw[->-=.5] (0,-.5) -- (1,-1);
\draw[fill=white] (0,.5) circle (.1);
\draw[fill=white] (0,-.5) circle (.1);
\node at (-1,1) [below]{$\scriptstyle{n}$};
\node at (-1,-1) [above]{$\scriptstyle{n}$};
\node at (1,1) [below]{$\scriptstyle{n}$};
\node at (1,-1) [above]{$\scriptstyle{n}$};
\node at (0,0) [right]{$\scriptstyle{j}$};
}\,\bigg\rangle_{\!3}$,
\item
$\bigg\langle\tikz[baseline=-.6ex, scale=0.4]{
\draw[-<-=.5] (0,0) -- (1,0);
\draw[triple={[line width=1.4pt, white] in [line width=2.2pt, black] in [line width=5.4pt, white]}]
(-1,1) -- (0,0);
\draw[-<-=.5] (-1,-1) -- (0,0);
\draw[-<-=.5] (1,0) -- (2,1);
\draw[triple={[line width=1.4pt, white] in [line width=2.2pt, black] in [line width=5.4pt, white]}]
(1,0) -- (2,-1);
\draw[fill=white] (0,0) circle (.1);
\draw[fill=white] (1,0) circle (.1);
\node at (-1,1) [below]{$\scriptstyle{s}$};
\node at (-1,-1) [above]{$\scriptstyle{n}$};
\node at (2,1) [below]{$\scriptstyle{n}$};
\node at (2,-1) [above]{$\scriptstyle{t}$};
\node at (.5,0) [above]{$\scriptstyle{n}$};
}\,\bigg\rangle_{\! 3}
=
\begin{Bmatrix}
n^{+}&s^{\pm}&n\\
n^{-}&t^{\mp}&n
\end{Bmatrix}
\bigg\langle\tikz[baseline=-.6ex, scale=0.4]{
\draw[-<-=.5] (0,-.5) -- (0,.5);
\draw[triple={[line width=1.4pt, white] in [line width=2.2pt, black] in [line width=5.4pt, white]}] 
(-1,1) -- (0,.5);
\draw[-<-=.5] (-1,-1) -- (0,-.5);
\draw[-<-=.5] (0,.5) -- (1,1);
\draw[triple={[line width=1.4pt, white] in [line width=2.2pt, black] in [line width=5.4pt, white]}] 
(0,-.5) -- (1,-1);
\draw[fill=white] (0,.5) circle (.1);
\draw[fill=white] (0,-.5) circle (.1);
\node at (-1,1) [below]{$\scriptstyle{s}$};
\node at (-1,-1) [above]{$\scriptstyle{n}$};
\node at (1,1) [below]{$\scriptstyle{n}$};
\node at (1,-1) [above]{$\scriptstyle{t}$};
\node at (0,0) [right]{$\scriptstyle{n}$};
}\,\bigg\rangle_{\!3}$,
\end{itemize}
where 
\[
\begin{Bmatrix}
n^{+}&s^{\pm}&n\\
n^{-}&t^{\mp}&n
\end{Bmatrix}
=
\frac{\theta(n,n,(i,i))\theta(n,n,(j,j))}
{\operatorname{Tet}\!
\begin{bmatrix}
n^{-}&n^{+}&(s,s)\\
n^{-}&n^{+}&(t,t)
\end{bmatrix}\Delta(n,0)}\, .
\]
\end{THM}
We call the coefficients {\em quantum $6j$ symbols}.
\begin{proof}
We prove this theorem in the same way as the proof of the Recoupling Theorem in \cite[Chapter~7]{KauffmanLins94}.
We can see that the clasped $A_2$ web space $W((i^{+},i^{-}),(n^{+},0),(n^{-},0))$ is one-dimensional vector space spanned by 
$\big\langle
%%%%%%%%%%
\tikz[baseline=-.6ex, scale=0.3]{
\draw[triple={[line width=1.4pt, white] in [line width=2.2pt, black] in [line width=5.4pt, white]}]
(0,0) -- (1,0);
\draw[-<-=.5] (-1,1) to[out=east, in=north west] (.0,.0);
\draw[->-=.5] (-1,-1) to[out=east, in=south west] (.0,.0);
\draw[fill=white] (0,0)  circle (.2);
\node at (-1,1) [left]{$\scriptstyle{n}$};
\node at (-1,-1) [left]{$\scriptstyle{n}$};
\node at (.5,0) [above]{$\scriptstyle{i}$};}
%%%%%%%%%%
\big\rangle_3$ 
by the similar argument in the proof of \cite[Lemma~3.3]{OhtsukiYamada97}.
A basis of $W((n^{+},0),(n^{-},0),(n^{+},0),(n^{-},0))$ is given as
\[
\mathcal{B}=\bigg\{\,
\Big\langle\tikz[baseline=-.6ex, scale=0.4]{
\draw[triple={[line width=1.4pt, white] in [line width=2.2pt, black] in [line width=5.4pt, white]}]
(0,0) -- (1,0);
\draw[-<-=.5] (-1,1) -- (0,0);
\draw[->-=.5] (-1,-1) -- (0,0);
\draw[-<-=.5] (1,0) -- (2,1);
\draw[->-=.5] (1,0) -- (2,-1);
\draw[fill=white] (0,0) circle (.1);
\draw[fill=white] (1,0) circle (.1);
\node at (-1,1) [below]{$\scriptstyle{n}$};
\node at (-1,-1) [above]{$\scriptstyle{n}$};
\node at (2,1) [below]{$\scriptstyle{n}$};
\node at (2,-1) [above]{$\scriptstyle{n}$};
\node at (.5,0) [above]{$\scriptstyle{i}$};
}\,\Big\rangle_{\! 3}
\,\Big\vert\,
0\leq i\leq n
\,\bigg\}.
\]
Firstly, 
we set a dotted line from south to north for any $A_2$ web $w$ in $W((n^{+},0),(n^{-},0),(n^{+},0),(n^{-},0))$ and replace two $A_2$ clasps on it with the $A_2$ clasp of type $(n,n)$ as follows:
\[w=\bigg\langle
%%%%%%%%%%
\tikz[baseline=-.6ex, scale=0.3]{
\draw[thick, dotted] (0,-2) -- (0,2);
\draw[-<-=.2] (135:2) to[out=east, in=west] (0,1)
to[out=east, in=north west] (1,0);
\draw[->-=.2] (-135:2) to[out=east, in=west] (0,-1)
to[out=east, in=south west] (1,0);
\draw[->-=.3] (45:2) -- (1,0);
\draw[-<-=.3] (-45:2) -- (1,0);
\draw (0,0) circle (2);
\draw[fill=lightgray!30] (1,0) circle (.5);
\draw[fill=white] (-.1,.6) rectangle (.1,1.4);
\draw[fill=white] (-.1,-.6) rectangle (.1,-1.4);
\draw[fill=cyan] (-45:2) circle (.1);
\draw[fill=cyan] (45:2) circle (.1);
\draw[fill=cyan] (135:2) circle (.1);
\draw[fill=cyan] (-135:2) circle (.1);
\node at (-45:2) [right]{$\scriptstyle{n}$};
\node at (45:2) [right]{$\scriptstyle{n}$};
\node at (135:2) [left]{$\scriptstyle{n}$};
\node at (-135:2) [left]{$\scriptstyle{n}$};
\node at (1,0) {$\scriptstyle{w}$};}
%%%%%%%%%%
\bigg\rangle_3\,
=\bigg\langle
%%%%%%%%%%
\tikz[baseline=-.6ex, scale=0.3]{
\draw[thick, dotted] (0,-2) -- (0,2);
\draw[-<-=.2] (135:2) to[out=east, in=west] (0,.2) -- (1,.2);
\draw[->-=.2] (-135:2) to[out=east, in=west] (0,-.2) -- (1,-.2);
\draw[->-=.3] (45:2) -- (1,0);
\draw[-<-=.3] (-45:2) -- (1,0);
\draw (0,0) circle (2);
\draw[fill=lightgray!30] (1,0) circle (.5);
\draw[fill=white] (-.1,-.4) rectangle (.1,.4);
\draw (-.1,0) -- (.1,0);
\draw[fill=cyan] (-45:2) circle (.1);
\draw[fill=cyan] (45:2) circle (.1);
\draw[fill=cyan] (135:2) circle (.1);
\draw[fill=cyan] (-135:2) circle (.1);
\node at (-45:2) [right]{$\scriptstyle{n}$};
\node at (45:2) [right]{$\scriptstyle{n}$};
\node at (135:2) [left]{$\scriptstyle{n}$};
\node at (-135:2) [left]{$\scriptstyle{n}$};
\node at (1,0) {$\scriptstyle{w}$};}
%%%%%%%%%%
\bigg\rangle_3\,
+w'.
\]
The definition of the $A_2$ clasp of type $(n,n)$ imply that $w'$ can be represented as a sum of $A_2$ webs such that the dotted line intersects with the webs at two clasps with color $k<n$. 
By repeating the replacement of clasps, 
\[w=\sum_{i=0}^na_i\bigg\langle
%%%%%%%%%%
\tikz[baseline=-.6ex, scale=0.3]{
\draw[triple={[line width=1.4pt, white] in [line width=2.2pt, black] in [line width=5.4pt, white]}]
(-.5,0) -- (1,0);
\draw[-<-=.5] (135:2) to[out=east, in=north west] (-.5,0);
\draw[->-=.5] (-135:2) to[out=east, in=south west] (-.5,0);
\draw[->-=.3] (45:2) -- (1,0);
\draw[-<-=.3] (-45:2) -- (1,0);
\draw[fill=white] (-.5,0) circle (.2);
\draw (0,0) circle (2);
\draw[fill=lightgray!30] (1,0) circle (.5);
\draw[fill=cyan] (-45:2) circle (.1);
\draw[fill=cyan] (45:2) circle (.1);
\draw[fill=cyan] (135:2) circle (.1);
\draw[fill=cyan] (-135:2) circle (.1);
\node at (-45:2) [right]{$\scriptstyle{n}$};
\node at (45:2) [right]{$\scriptstyle{n}$};
\node at (135:2) [left]{$\scriptstyle{n}$};
\node at (-135:2) [left]{$\scriptstyle{n}$};
\node at (0,0) [above]{$\scriptstyle{i}$};
\node at (1,0) {$\scriptscriptstyle{w_i}$};}
%%%%%%%%%%
\bigg\rangle_3\,
=\sum_{i=0}^na_i\bigg\langle
%%%%%%%%%%
\tikz[baseline=-.6ex, scale=0.3]{
\draw[triple={[line width=1.4pt, white] in [line width=2.2pt, black] in [line width=5.4pt, white]}]
(-.5,0) -- (.5,0);
\draw[-<-=.5] (135:2) to[out=east, in=north west] (-.5,0);
\draw[->-=.5] (-135:2) to[out=east, in=south west] (-.5,0);
\draw[->-=.5] (45:2) to[out=west, in=north east] (.5,0);
\draw[-<-=.5] (-45:2) to[out=west, in=south east] (.5,0);
\draw[fill=white] (-.5,0) circle (.2);
\draw[fill=white] (.5,0) circle (.2);
\draw (0,0) circle (2);
\draw[fill=cyan] (-45:2) circle (.1);
\draw[fill=cyan] (45:2) circle (.1);
\draw[fill=cyan] (135:2) circle (.1);
\draw[fill=cyan] (-135:2) circle (.1);
\node at (-45:2) [right]{$\scriptstyle{n}$};
\node at (45:2) [right]{$\scriptstyle{n}$};
\node at (135:2) [left]{$\scriptstyle{n}$};
\node at (-135:2) [left]{$\scriptstyle{n}$};
\node at (0,0) [above]{$\scriptstyle{i}$};}
%%%%%%%%%%
\bigg\rangle_3\, .
\]
By taking a dotted line from west to east, 
we also obtain another basis
\[
\mathcal{B}'=\bigg\{\,
\Big\langle\tikz[baseline=-.6ex, scale=0.4]{
\draw[triple={[line width=1.4pt, white] in [line width=2.2pt, black] in [line width=5.4pt, white]}]
(0,-.5) -- (0,.5);
\draw[-<-=.5] (-1,1) -- (0,.5);
\draw[->-=.5] (-1,-1) -- (0,-.5);
\draw[-<-=.5] (0,.5) -- (1,1);
\draw[->-=.5] (0,-.5) -- (1,-1);
\draw[fill=white] (0,.5) circle (.1);
\draw[fill=white] (0,-.5) circle (.1);
\node at (-1,1) [below]{$\scriptstyle{n}$};
\node at (-1,-1) [above]{$\scriptstyle{n}$};
\node at (1,1) [below]{$\scriptstyle{n}$};
\node at (1,-1) [above]{$\scriptstyle{n}$};
\node at (0,0) [right]{$\scriptstyle{j}$};
}\,\Big\rangle_{\!3}
\,\Big\vert\,
0\leq j\leq n
\,\bigg\}\, .
\]
Therefore, 
the first recoupling formula is the change of basis from $\mathcal{B}$ to $\mathcal{B}'$. 
The value of a quantum $6j$ symbol
$
\begin{Bmatrix}
n^{-}&n^{+}&(j_0,j_0)\\
n^{-}&n^{+}&(i,i)
\end{Bmatrix}
$ 
is obtained by closing $A_2$ webs in the both sides using $
\Big\langle\tikz[baseline=-.6ex, scale=0.4]{
\draw[triple={[line width=1.4pt, white] in [line width=2.2pt, black] in [line width=5.4pt, white]}]
(0,-.5) -- (0,.5);
\draw[-<-=.5] (-1,1) -- (0,.5);
\draw[->-=.5] (-1,-1) -- (0,-.5);
\draw[-<-=.5] (0,.5) -- (1,1);
\draw[->-=.5] (0,-.5) -- (1,-1);
\draw[fill=white] (0,.5) circle (.1);
\draw[fill=white] (0,-.5) circle (.1);
\node at (-1,1) [below]{$\scriptstyle{n}$};
\node at (-1,-1) [above]{$\scriptstyle{n}$};
\node at (1,1) [below]{$\scriptstyle{n}$};
\node at (1,-1) [above]{$\scriptstyle{n}$};
\node at (0,0) [right]{$\scriptstyle{j_0}$};
}\,\Big\rangle_{\!3}
$ and using Lemma~\ref{bigon}.

Next, 
we prove the second formula.
An argument about the Euler number of a polygon developed in \cite{OhtsukiYamada97} shows the clasped $A_2$ web space $W((i^{+},j^{-}),(s^{+},s^{-}),(n^{+},0))$ is non-zero if and only if $(i,j)=(0,n)$. 
For any $A_2$ web in $W((t^{-},t^{+}),(n^{-},0),(s^{+},s^{-}),(n^{+},0))$, 
we take a dotted line as the first proof and replace intersecting webs with clasp of type $(i,j)$ for some $i$ and $j$. 
However, 
the web vanish except in the case of $(i,j)=(0,n)$. 
Thus, 
$W((t^{-},t^{+}),(n^{-},0),(s^{+},s^{-}),(n^{+},0))$ is the one-dimensional vector space and we can take its basis as 
$\bigg\{\Big\langle\tikz[baseline=-.6ex, scale=0.4]{
\draw[-<-=.5] (0,0) -- (1,0);
\draw[triple={[line width=1.4pt, white] in [line width=2.2pt, black] in [line width=5.4pt, white]}]
(-1,1) -- (0,0);
\draw[-<-=.5] (-1,-1) -- (0,0);
\draw[-<-=.5] (1,0) -- (2,1);
\draw[triple={[line width=1.4pt, white] in [line width=2.2pt, black] in [line width=5.4pt, white]}]
(1,0) -- (2,-1);
\draw[fill=white] (0,0) circle (.1);
\draw[fill=white] (1,0) circle (.1);
\node at (-1,1) [below]{$\scriptstyle{s}$};
\node at (-1,-1) [above]{$\scriptstyle{n}$};
\node at (2,1) [below]{$\scriptstyle{n}$};
\node at (2,-1) [above]{$\scriptstyle{t}$};
\node at (.5,0) [above]{$\scriptstyle{n}$};
}\,\Big\rangle_{\! 3}\bigg\}
$
and
$\bigg\{
\Big\langle\tikz[baseline=-.6ex, scale=0.4]{
\draw[-<-=.5] (0,-.5) -- (0,.5);
\draw[triple={[line width=1.4pt, white] in [line width=2.2pt, black] in [line width=5.4pt, white]}] 
(-1,1) -- (0,.5);
\draw[-<-=.5] (-1,-1) -- (0,-.5);
\draw[-<-=.5] (0,.5) -- (1,1);
\draw[triple={[line width=1.4pt, white] in [line width=2.2pt, black] in [line width=5.4pt, white]}] 
(0,-.5) -- (1,-1);
\draw[fill=white] (0,.5) circle (.1);
\draw[fill=white] (0,-.5) circle (.1);
\node at (-1,1) [below]{$\scriptstyle{s}$};
\node at (-1,-1) [above]{$\scriptstyle{n}$};
\node at (1,1) [below]{$\scriptstyle{n}$};
\node at (1,-1) [above]{$\scriptstyle{t}$};
\node at (0,0) [right]{$\scriptstyle{n}$};
}\,\Big\rangle_{\!3}\bigg\}$. 
Let us evaluate the $6j$ symbol 
$
\begin{Bmatrix}
n^{+}&s^{\pm}&n\\
n^{-}&t^{\mp}&n
\end{Bmatrix}
$ by closing $A_2$ webs in the both sides using 
$
\Big\langle\tikz[baseline=-.6ex, scale=0.4, xscale=-1]{
\draw[->-=.5] (0,0) -- (1,0);
\draw[->-=.5] (1,0) -- (2,1);
\draw[triple={[line width=1.4pt, white] in [line width=2.2pt, black] in [line width=5.4pt, white]}]
(1,0) -- (2,-1);
\draw[fill=white] (1,0) circle (.1);
\node at (2,1) [below]{$\scriptstyle{n}$};
\node at (2,-1) [above]{$\scriptstyle{t}$};
\node at (.5,0) [above]{$\scriptstyle{n}$};
}\,\Big\rangle_{\! 3}
$. 
By Lemma~\ref{bibigon},
\begin{align*}
\frac{\theta(n,n,(t,t))}{\Delta(n,0)}
\Big\langle\tikz[baseline=-.6ex, scale=0.4]{
\draw[-<-=.5] (0,0) -- (1,0);
\draw[triple={[line width=1.4pt, white] in [line width=2.2pt, black] in [line width=5.4pt, white]}]
(-1,1) -- (0,0);
\draw[-<-=.5] (-1,-1) -- (0,0);
\draw[fill=white] (0,0) circle (.1);
\node at (-1,1) [below]{$\scriptstyle{s}$};
\node at (-1,-1) [above]{$\scriptstyle{n}$};
\node at (.5,0) [above]{$\scriptstyle{n}$};
}\,\Big\rangle_{\! 3}
&=
\begin{Bmatrix}
n^{+}&s^{\pm}&n\\
n^{-}&t^{\mp}&n
\end{Bmatrix}
\Big\langle\tikz[baseline=-.6ex, scale=0.4]{
\draw[-<-=.5] (0,-.5) -- (0,.5);
\draw[triple={[line width=1.4pt, white] in [line width=2.2pt, black] in [line width=5.4pt, white]}] 
(-1,1) -- (0,.5);
\draw[-<-=.5] (-1,-1) -- (0,-.5);
\draw[triple={[line width=1.4pt, white] in [line width=2.2pt, black] in [line width=5.4pt, white]}] 
(0,-.5) to[out=east, in=south west] (2,0);
\draw[-<-=.5] (0,.5) to[out=east, in=north west] (2,0);
\draw[-<-=.5] (2,0) -- (3,0);
\draw[fill=white] (2,0) circle (.1);
\draw[fill=white] (0,.5) circle (.1);
\draw[fill=white] (0,-.5) circle (.1);
\node at (-1,1) [below]{$\scriptstyle{s}$};
\node at (-1,-1) [above]{$\scriptstyle{n}$};
\node at (2.5,0) [above]{$\scriptstyle{n}$};
\node at (1,1) {$\scriptstyle{n}$};
\node at (1,-1) {$\scriptstyle{t}$};
\node at (0,0) [right]{$\scriptstyle{n}$};
}\,\Big\rangle_{\!3}.
\end{align*}
We apply the first recoupling formula to the $A_2$ web in the right-hand side and Lemma~\ref{bigon}. 
Thus, 
we obtain
\[
\frac{\theta(n,n,(t,t))}{\Delta(n,0)}
=
\begin{Bmatrix}
n^{+}&s^{\pm}&n\\
n^{-}&t^{\mp}&n
\end{Bmatrix}
\begin{Bmatrix}
n^{-}&n^{+}&(s,s)\\
n^{-}&n^{+}&(t,t)
\end{Bmatrix}
\frac{\theta(n,n,(s,s))}{\Delta(s,s)}.
\]
\end{proof}

\section{Triangle-free basis for $W((n^{+},0),(n^{-},0),\dots,(n^{+},0),(n^{-},0))$}
In this section, 
we construct a basis of a clasped $A_2$ web space $W(n^{+},n^{-},\dots,n^{+},n^{-})$ from a certain type of triangulations called {\em triangle-free triangulations}. 

Let $(D,Q_k)$ be a disk with marked points $Q_k=\{\,q_0,q_1,\dots,q_{k-1}\,\}$. 
We call a part of the boundary of $D$ between $q_i$ and $q_{i+1}$ an {\em edge} of $(D,Q_k)$.
A {\em triangulation} of $(D,Q_k)$ is a set of $(k-3)$ simple arcs ended at $Q_k$ such that these arcs are not isotopic to other arcs and edges of $(D,Q_k)$. 
A {\em chord} is an arc in a triangulation and {\em short chord} a chord connecting $q_{i-1}$ and $q_{i+1}$ where $1\leq k\leq k-1$ and $q_k=q_0$. 
\begin{DEF}[Triangle-free triangulation~\cite{AdinFirerRoichman10}]
A triangulation of $(D,Q_k)$ is {\em triangle-free} if any triple of chords does not make a triangle.
\end{DEF}

We remark that a triangulation of $(D,Q_k)$ is triangle-free for $k>3$ if and only if it contains only two short chords, see \cite[Claim~2.3]{AdinFirerRoichman10}. 
A {\em proper coloring} of a triangle-free triangulation is a color of the chords with $\{\,0,1,\dots, k-4\,\}$ labeled by the following rules:
\begin{enumerate}
\item we choose one of the two short chords and label it by $0$,
\item if a triangle contains two chords, then these chords are labeled by $i$ and $i+1$.
\end{enumerate}
See, for example, Fig.~\ref{Fig:tfree}.
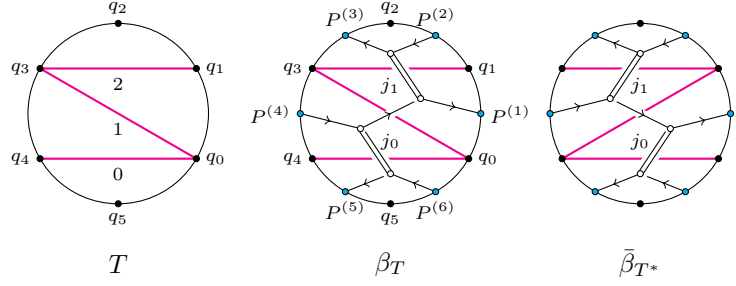
\begin{figure}
\begin{tikzpicture}[baseline=-.6ex, scale=0.4]
\draw[thick, magenta] (30:3) -- (150:3);
\draw[thick, magenta] (150:3) -- (-30:3);
\draw[thick, magenta] (-30:3) -- (-150:3);
\draw (0,0) circle (3);
\draw[fill=black] (30:3)  circle (.1);
\draw[fill=black] (90:3)  circle (.1);
\draw[fill=black] (150:3)  circle (.1);
\draw[fill=black] (-30:3)  circle (.1);
\draw[fill=black] (-90:3)  circle (.1);
\draw[fill=black] (-150:3)  circle (.1);
\node at (30:3) [right]{$\scriptstyle{q_1}$};
\node at (90:3) [above]{$\scriptstyle{q_2}$};
\node at (150:3) [left]{$\scriptstyle{q_3}$};
\node at (-30:3) [right]{$\scriptstyle{q_0}$};
\node at (-90:3) [below]{$\scriptstyle{q_5}$};
\node at (-150:3) [left]{$\scriptstyle{q_4}$};
\node at (0,-1.5) [below]{$\scriptstyle{0}$};
\node at (0,0) [below]{$\scriptstyle{1}$};
\node at (0,1.5) [below]{$\scriptstyle{2}$};
\node at (-90:5) {Triangle-free};
\end{tikzpicture}
\begin{tikzpicture}[baseline=-.6ex, scale=0.4]
\draw[thick, magenta] (-30:3) -- (90:3);
\draw[thick, magenta] (150:3) -- (-30:3);
\draw[thick, magenta] (-30:3) -- (-150:3);
\draw (0,0) circle (3);
\draw[fill=black] (30:3)  circle (.1);
\draw[fill=black] (90:3)  circle (.1);
\draw[fill=black] (150:3)  circle (.1);
\draw[fill=black] (-30:3)  circle (.1);
\draw[fill=black] (-90:3)  circle (.1);
\draw[fill=black] (-150:3)  circle (.1);
\node at (30:3) [right]{$\scriptstyle{q_1}$};
\node at (90:3) [above]{$\scriptstyle{q_2}$};
\node at (150:3) [left]{$\scriptstyle{q_3}$};
\node at (-30:3) [right]{$\scriptstyle{q_0}$};
\node at (-90:3) [below]{$\scriptstyle{q_5}$};
\node at (-150:3) [left]{$\scriptstyle{q_4}$};
\node at (0,-1.5) [above]{$\scriptstyle{2}$};
\node at (0,0) [above]{$\scriptstyle{1}$};
\node at (0,1.5) [right]{$\scriptstyle{0}$};
\node at (-90:5) {Triangle-free};
\end{tikzpicture}
\begin{tikzpicture}[baseline=-.6ex, scale=0.4]
\draw[thick, magenta] (-30:3) -- (90:3);
\draw[thick, magenta] (90:3) -- (-150:3);
\draw[thick, magenta] (-30:3) -- (-150:3);
\draw (0,0) circle (3);
\draw[fill=black] (30:3)  circle (.1);
\draw[fill=black] (90:3)  circle (.1);
\draw[fill=black] (150:3)  circle (.1);
\draw[fill=black] (-30:3)  circle (.1);
\draw[fill=black] (-90:3)  circle (.1);
\draw[fill=black] (-150:3)  circle (.1);
\node at (30:3) [right]{$\scriptstyle{q_1}$};
\node at (90:3) [above]{$\scriptstyle{q_2}$};
\node at (150:3) [left]{$\scriptstyle{q_3}$};
\node at (-30:3) [right]{$\scriptstyle{q_0}$};
\node at (-90:3) [below]{$\scriptstyle{q_5}$};
\node at (-150:3) [left]{$\scriptstyle{q_4}$};
\node at (-90:5) {Non triangle-free};
\end{tikzpicture}
\caption{Colored triangle-free triangulations for $(D,Q_6)$}
\label{Fig:tfree}
\end{figure}

Let us take a colored triangle-free triangulation $T$ of $(D,Q_{2k})$.
We take another set of marked points $P=P^{(1)}\cup P^{(2)}\cup \dots P^{(2k)}$ where $P^{(i)}=\{\,p_1^{(i)}, p_2^{(i)},\dots, p_n^{(i)}\,\}$ for $0 \leq i\leq 2k$. 
The marked points are arranged as $q_0<p_{i_1}^{(1)}<q_1<p_{i_2}^{(2)}<\dots<q_{2k-1}<p_{i_{2k}}^{(2k)}<q_0$ with respect to the cyclic order induced from the orientation of $\partial D$. 
We set the sign of marked points in $P^{(i)}$ is $\epsilon_i=(-1)^{i+1}$. 
Let us consider a clasped $A_2$ web space $W((n^{+},0),(n^{-},0),\dots,(n^{+},0),(n^{-},0))$ of $(D,P)$ such that the webs adjacent to $P^{(i)}$ has an $A_2$ clasp of type $(n^{\epsilon_i},0)$. 
We denote the clasped $A_2$ web space by $W_{2k}(n^\pm)$ for simplicity.

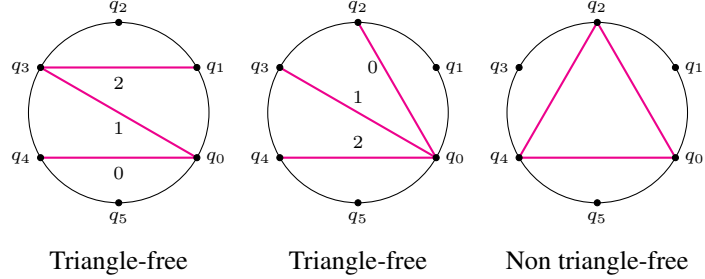
\begin{figure}
\begin{tikzpicture}[baseline=-.6ex, scale=0.4]
\draw[thick, magenta] (30:3) -- (150:3);
\draw[thick, magenta] (150:3) -- (-30:3);
\draw[thick, magenta] (-30:3) -- (-150:3);
\draw (0,0) circle (3);
\draw[fill=black] (30:3)  circle (.1);
\draw[fill=black] (90:3)  circle (.1);
\draw[fill=black] (150:3)  circle (.1);
\draw[fill=black] (-30:3)  circle (.1);
\draw[fill=black] (-90:3)  circle (.1);
\draw[fill=black] (-150:3)  circle (.1);
\node at (30:3) [right]{$\scriptstyle{q_1}$};
\node at (90:3) [above]{$\scriptstyle{q_2}$};
\node at (150:3) [left]{$\scriptstyle{q_3}$};
\node at (-30:3) [right]{$\scriptstyle{q_0}$};
\node at (-90:3) [below]{$\scriptstyle{q_5}$};
\node at (-150:3) [left]{$\scriptstyle{q_4}$};
\node at (0,-1.5) [below]{$\scriptstyle{0}$};
\node at (0,0) [below]{$\scriptstyle{1}$};
\node at (0,1.5) [below]{$\scriptstyle{2}$};
\node at (-90:5) {$T$};
\end{tikzpicture}
\begin{tikzpicture}[baseline=-.6ex, scale=0.4]
\draw[thick, magenta] (30:3) -- (150:3);
\draw[thick, magenta] (150:3) -- (-30:3);
\draw[thick, magenta] (-30:3) -- (-150:3);
\draw[-<-=.5] (240:3) -- (0,-2);
\draw[->-=.5] (300:3) -- (0,-2);
\draw[triple={[line width=1.4pt, white] in [line width=2.2pt, black] in [line width=5.4pt, white]}] (0,-2) -- (-1,-.5);
\draw[->-=.5] (180:3) -- (-1,-.5);
\draw[->-=.5, white, double=black, double distance=0.4pt, ultra thick] (-1,-.5) -- (1,.5);
\draw[-<-=.5] (0:3) -- (1,.5);
\draw[triple={[line width=1.4pt, white] in [line width=2.2pt, black] in [line width=5.4pt, white]}] (1,.5) -- (0,2);
\draw[->-=.5] (60:3) -- (0,2);
\draw[-<-=.5] (120:3) -- (0,2);
\draw (0,0) circle (3);
\draw[fill=white] (0,-2) circle (.1);
\draw[fill=white] (-1,-.5) circle (.1);
\draw[fill=white] (1,.5) circle (.1);
\draw[fill=white] (0,2) circle (.1);
\draw[fill=cyan] (0:3)  circle (.1);
\draw[fill=cyan] (60:3)  circle (.1);
\draw[fill=cyan] (120:3)  circle (.1);
\draw[fill=cyan] (180:3)  circle (.1);
\draw[fill=cyan] (240:3)  circle (.1);
\draw[fill=cyan] (300:3)  circle (.1);
\node at (0:3) [right]{$\scriptstyle{P^{(1)}}$};
\node at (60:3) [above]{$\scriptstyle{P^{(2)}}$};
\node at (120:3) [above]{$\scriptstyle{P^{(3)}}$};
\node at (180:3) [left]{$\scriptstyle{P^{(4)}}$};
\node at (240:3) [below]{$\scriptstyle{P^{(5)}}$};
\node at (300:3) [below]{$\scriptstyle{P^{(6)}}$};
\draw[fill=black] (30:3)  circle (.1);
\draw[fill=black] (90:3)  circle (.1);
\draw[fill=black] (150:3)  circle (.1);
\draw[fill=black] (-30:3)  circle (.1);
\draw[fill=black] (-90:3)  circle (.1);
\draw[fill=black] (-150:3)  circle (.1);
\node at (30:3) [right]{$\scriptstyle{q_1}$};
\node at (90:3) [above]{$\scriptstyle{q_2}$};
\node at (150:3) [left]{$\scriptstyle{q_3}$};
\node at (-30:3) [right]{$\scriptstyle{q_0}$};
\node at (-90:3) [below]{$\scriptstyle{q_5}$};
\node at (-150:3) [left]{$\scriptstyle{q_4}$};
\node at (0,-1.5) [above]{$\scriptstyle{j_0}$};
\node at (0,1.5) [below]{$\scriptstyle{j_1}$};
\node at (-90:5) {$\beta_T$};
\end{tikzpicture}
\begin{tikzpicture}[baseline=-.6ex, scale=0.4, xscale=-1]
\draw[thick, magenta] (30:3) -- (150:3);
\draw[thick, magenta] (150:3) -- (-30:3);
\draw[thick, magenta] (-30:3) -- (-150:3);
\draw[->-=.5] (240:3) -- (0,-2);
\draw[-<-=.5] (300:3) -- (0,-2);
\draw[triple={[line width=1.4pt, white] in [line width=2.2pt, black] in [line width=5.4pt, white]}] (0,-2) -- (-1,-.5);
\draw[-<-=.5] (180:3) -- (-1,-.5);
\draw[-<-=.5, white, double=black, double distance=0.4pt, ultra thick] (-1,-.5) -- (1,.5);
\draw[->-=.5] (0:3) -- (1,.5);
\draw[triple={[line width=1.4pt, white] in [line width=2.2pt, black] in [line width=5.4pt, white]}] (1,.5) -- (0,2);
\draw[-<-=.5] (60:3) -- (0,2);
\draw[->-=.5] (120:3) -- (0,2);
\draw (0,0) circle (3);
\draw[fill=white] (0,-2) circle (.1);
\draw[fill=white] (-1,-.5) circle (.1);
\draw[fill=white] (1,.5) circle (.1);
\draw[fill=white] (0,2) circle (.1);
\draw[fill=cyan] (0:3)  circle (.1);
\draw[fill=cyan] (60:3)  circle (.1);
\draw[fill=cyan] (120:3)  circle (.1);
\draw[fill=cyan] (180:3)  circle (.1);
\draw[fill=cyan] (240:3)  circle (.1);
\draw[fill=cyan] (300:3)  circle (.1);
\draw[fill=black] (30:3)  circle (.1);
\draw[fill=black] (90:3)  circle (.1);
\draw[fill=black] (150:3)  circle (.1);
\draw[fill=black] (-30:3)  circle (.1);
\draw[fill=black] (-90:3)  circle (.1);
\draw[fill=black] (-150:3)  circle (.1);
\node at (0,-1.5) [above]{$\scriptstyle{j_0}$};
\node at (0,1.5) [below]{$\scriptstyle{j_1}$};
\node at (-90:5) {$\bar{\beta}_{T^*}$};
\end{tikzpicture}
\caption{Colored triangle-free triangulations for $(D,Q_6)$}
\label{Fig:tfreebasis}
\end{figure}

\begin{DEF}[Triangle-free basis]\label{tfreebasis}
A {\em triangle-free basis web} $\beta_{T}(j_0,j_1,\dots,j_{k-2})$ is a clasped $A_2$ web obtained by patching the following pieces of $A_2$ webs: 
\begin{enumerate}
\item 
For a triangle containing two chords colored with $2i$ and $2i+1$,\\
%%%%%%%%%%
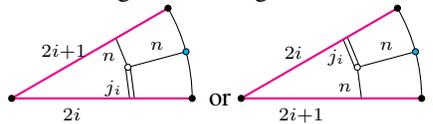
\begin{tikzpicture}[baseline=-.6ex, scale=0.4]
\draw[triple={[line width=1.4pt, white] in [line width=2.2pt, black] in [line width=5.4pt, white]}]
(15:4) to (0:4);
\draw (15:4) -- (15:6);
\draw (15:4) -- (30:4);
\draw[thick, magenta] (0:0) -- (0:6);
\draw[thick, magenta] (0:0) -- (30:6);
\draw (0:6) arc (0:30:6);
\draw[fill=white] (15:4)  circle (.1);
\draw[fill=cyan] (15:6)  circle (.1);
\draw[fill=black] (0:6)  circle (.1);
\draw[fill=black] (30:6)  circle (.1);
\draw[fill=black] (0:0)  circle (.1);
\node at (20:4) [left]{$\scriptstyle{n}$};
\node at (15:5) [above]{$\scriptstyle{n}$};
\node at (5:4) [left]{$\scriptstyle{j_i}$};
\node at (30:2) [above]{$\scriptstyle{2i+1}$};
\node at (0:2) [below]{$\scriptstyle{2i}$};
\end{tikzpicture}
%%%%%%%%%%
\ or\ 
%%%%%%%%%%
\begin{tikzpicture}[baseline=-.6ex, scale=0.4]
\draw (15:4) -- (0:4);
\draw (15:4) -- (15:6);
\draw[triple={[line width=1.4pt, white] in [line width=2.2pt, black] in [line width=5.4pt, white]}] 
(15:4) -- (30:4);
\draw[thick, magenta] (0:0) -- (0:6);
\draw[thick, magenta] (0:0) -- (30:6);
\draw (0:6) arc (0:30:6);
\draw[fill=white] (15:4)  circle (.1);
\draw[fill=cyan] (15:6)  circle (.1);
\draw[fill=black] (0:6)  circle (.1);
\draw[fill=black] (30:6)  circle (.1);
\draw[fill=black] (0:0)  circle (.1);
\node at (20:4) [left]{$\scriptstyle{j_i}$};
\node at (15:5) [above]{$\scriptstyle{n}$};
\node at (5:4) [left]{$\scriptstyle{n}$};
\node at (30:2) [above]{$\scriptstyle{2i}$};
\node at (0:2) [below]{$\scriptstyle{2i+1}$};
%%%%%%%%%%
\end{tikzpicture}
.
\item
For a triangle containing only one chord colored by $2i=0,2k-4$,\\
%%%%%%%%%%
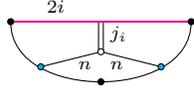
\begin{tikzpicture}[baseline=-.6ex, scale=0.4]
\draw (3,1) -- (1,.5);
\draw (3,1) -- (5,.5);
\draw[triple={[line width=1.4pt, white] in [line width=2.2pt, black] in [line width=5.4pt, white]}] 
(3,1) -- (3,2);
\draw[thick, magenta] (0,2) -- (6,2);
\draw (0,2) to[out=south, in=west] (3,0) to[out=east, in=south] (6,2);
\draw[fill=white] (3,1)  circle (.1);
\draw[fill=cyan] (1,.5)  circle (.1);
\draw[fill=cyan] (5,.5)  circle (.1);
\draw[fill=black] (0,2)  circle (.1);
\draw[fill=black] (6,2)  circle (.1);
\draw[fill=black] (3,0)  circle (.1);
\node at (3,1) [above right]{$\scriptstyle{j_i}$};
\node at (3,1) [below right]{$\scriptstyle{n}$};
\node at (3,1) [below left]{$\scriptstyle{n}$};
\node at (1.5,2) [above]{$\scriptstyle{2i}$};
%%%%%%%%%%
\end{tikzpicture}
.
\end{enumerate}
A {\em triangle-free basis} of $W_{2k}(n^\pm)$ with respect to $T$ is $\mathcal{B}_T=\{\,\beta_T(j_0,j_1,\dots,j_{k-2}) \mid 0 \leq j_i \leq n\,\}$, see Fig.~\ref{Fig:tfreebasis}
\end{DEF}

\begin{THM}
$\mathcal{B}_T$ is a basis of $W_{2k}(n^\pm)$ for any colored triangle-free triangulation $T$.
\end{THM}

\begin{proof}
Fix a colored triangle-free triangulation $T$ and take any $A_2$ web $w$. 
In a manner similar to the proof of Theorem~\ref{recoupling}, 
The $A_2$ webs intersecting the short chord with color $0$ can be replaced by a sum of $A_2$ clasps. 
Then, 
$w$ is represented as a sum of $A_2$ webs containing trivalent graphs in (2) of Definition~\ref{tfreebasis}. 
Next, 
we consider $A_2$ webs in the triangle containing chords with color $0$ and $1$. 
Then, 
these $A_2$ webs are replaced by trivalent graph in (1) of Definition~\ref{tfreebasis}. 
In such a way, 
we can represent $w$ by a sum of triangle-free basis webs. 
Assume that $\sum_{0\leq j_i \leq n}a_{j_0,j_1,\dots,j_{k-2}}\beta_T(j_0,j_1,\dots,j_{k-2})=0$. 
Let $T^*$ be the mirror image of the colored triangle-free triangulation $T$. 
For any clasped $A_2$ web $w$ in $W_{2k}(n^\pm)$, 
$\bar{w}$ denote a clasped $A_2$ web in $W_{2k}(n^\mp)$ obtained by reversing the direction of $w$. 
It is easy to show that the pairing $\langle\,\beta_T(j_0,j_1,\dots,j_{k-2})\mid\bar{\beta}_{T^*}(s_0,s_1,\dots,s_{k-2})\,\rangle_3$ is non-zero if and only if $j_i=s_i$ for all $i=0,1,\dots, k-2$. 
This pairing is a map obtained by connecting each endpoint of $\beta_T(j_0,j_1,\dots,j_{k-2})$ and the endpoint of $\bar{\beta}_{T^*}(s_0,s_1,\dots,s_{k-2})$ in its mirror point. 
We can prove by induction on $k$ and using Lemma~\ref{bigon} and \ref{bibigon}.
\end{proof}

\section{A representation of the Pure Braid Group $P_{2k}$}
In this section, 
we give a definition of representations $\rho_n$ of the pure braid group of $2k$ strands and compute it for all standard generators. 

\subsection{Definition of $\rho_n$}
Let $T$ be a colored triangle-free triangulation of $(D,Q_{2k})$ and $\mathcal{B}_T$ the associated triangle-free basis of $W_{2k}(n^{\pm})$. 
In this section, 
we take an orientation of $\partial D$ in the clockwise direction for convenience. 
In fact, 
$q_0,q_1,\dots,q_{2k-1}$ and $P^{(1)}, P^{(2)},\dots,P^{(2k)}$ are arranged clockwise.
Moreover, 
we depict $D$ square such that $q_0$ is in the bottom side and other points of $Q_{2k}$ and $P$ in the upper side, see below:\\
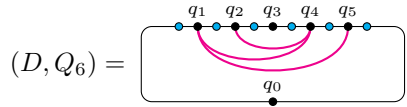
\begin{center}
\begin{tikzpicture}[baseline=-.6ex, scale=0.5]
\draw[rounded corners] (0,-2) rectangle (7,0);
\draw[thick, magenta] (1.5,0) to[out=south, in=south] (4.5,0);
\draw[thick, magenta] (2.5,0) to[out=south, in=south] (4.5,0);
\draw[thick, magenta] (1.5,0) to[out=south, in=south] (5.5,0);
\draw[fill=cyan] (1,0)  circle (.1);
\draw[fill=cyan] (2,0)  circle (.1);
\draw[fill=cyan] (3,0)  circle (.1);
\draw[fill=cyan] (4,0)  circle (.1);
\draw[fill=cyan] (5,0)  circle (.1);
\draw[fill=cyan] (6,0)  circle (.1);
\draw[fill=black] (3.5,-2)  circle (.1);
\draw[fill=black] (1.5,0)  circle (.1);
\draw[fill=black] (2.5,0)  circle (.1);
\draw[fill=black] (3.5,0)  circle (.1);
\draw[fill=black] (4.5,0)  circle (.1);
\draw[fill=black] (5.5,0)  circle (.1);
\node at (3.5,-2) [above]{$\scriptstyle{q_0}$};
\node at (1.5,0) [above]{$\scriptstyle{q_1}$};
\node at (2.5,0) [above]{$\scriptstyle{q_2}$};
\node at (3.5,0) [above]{$\scriptstyle{q_3}$};
\node at (4.5,0) [above]{$\scriptstyle{q_4}$};
\node at (5.5,0) [above]{$\scriptstyle{q_5}$};
\node at (0,-1) [left]{$(D,Q_6)=\ $};
\end{tikzpicture}
\end{center}
The braid group of $m$ strands has a presentation~\cite{Artin47}
\[
B_m=\Bigg\langle\, \sigma_1,\sigma_2,\dots,\sigma_{m-1}\,\Bigg\vert\,
\begin{aligned} 
\sigma_i\sigma_{i+1}\sigma_i&=\sigma_{i+1}\sigma_i\sigma_{i+1}\quad(1\leq i< m-1)\\
\sigma_i\sigma_j&=\sigma_j\sigma_i\quad(\left|i-j\right|>1)
\end{aligned}
\,\Bigg\rangle. 
\]
The generator $\sigma_i$ represents the following braid diagram:\\
\begin{center}
\begin{tikzpicture}[baseline=-.6ex, scale=0.5]
\draw[white, double=black, double distance=0.4pt, ultra thick] 
(.5,-1) to[out=north, in=south] (-.5,1);
\draw[white, double=black, double distance=0.4pt, ultra thick] 
(-.5,-1) to[out=north, in=south] (.5,1);
\draw (-1,-1) -- (-1,1);
\draw (-3,-1) -- (-3,1);
\draw (1,-1) -- (1,1);
\draw (3,-1) -- (3,1);
\node at (-2,0) {$\cdots$};
\node at (2,0) {$\cdots$};
\node at (-3,0) [left]{$\sigma_i=\ $};
\node at (-.5,1) [above]{$\scriptstyle{i}$};
\node at (.5,1) [above]{$\scriptstyle{i+1}$};
\end{tikzpicture}
\ for $i=1,2,\dots,m-1$.
\end{center}
The composition of two braid diagrams $ab$ is given by gluing $b$ on the top of $a$.
The pure braid group $P_m$ is the kernel of the map $B_m\to S_m$ where $S_m$ is the symmetric group. 
It is well-known that a generating set of $P_m$ is given by
\[
\{\,
A_{i,j}=\sigma_{j-1}^{-1}\sigma_{j-2}^{-1}\dots\sigma_{i+1}^{-1}\sigma_i^2\sigma_{i+1}\dots\sigma_{j-2}\sigma_{j-1}
\mid 1\leq i<j \leq m
\,\}
\]
 and see \cite{Birman74} for detail on the presentation.
The braid diagram of the generator $A_{i,j}$ is the following:
\begin{center}
\begin{tikzpicture}[baseline=-.6ex, scale=0.5]
\draw (-2,-1) -- (-2,1);
\draw (-1,-1) -- (-1,1);
\draw (0,-1) -- (0,1);
\draw (1,-1) -- (1,1);
\draw (-.2,.0) to[out=north, in=south] (-.5,.3) -- (-.5,1);
%%%
\draw[white, double=black, double distance=0.4pt, ultra thick] (1.5,-1) to[out=north, in=east] (0,-.3);
\draw[white, double=black, double distance=0.4pt, ultra thick] (0,-.3) to[out=west, in=south] (-.5,0);
\draw[white, double=black, double distance=0.4pt, ultra thick] (-.5,0) to[out=north, in=west] (0,.3);
\draw[white, double=black, double distance=0.4pt, ultra thick] (-.5,0) to[out=north, in=west] (0,.3) to[out=east, in=south] (1.5,1);
%%%
\draw[white, double=black, double distance=0.4pt, ultra thick]  (-.5,-1) -- (-.5,-.3) to[out=north, in=south] (-.2,.0);
\draw (2,-1) -- (2,1);
\draw (3,-1) -- (3,1);
\node at (-1.5,0) {$\scriptstyle{\cdots}$};
\node at (.5,0) {$\scriptstyle{\cdots}$};
\node at (2.5,0) {$\scriptstyle{\cdots}$};
\node at (-2,0) [left]{$A_{i,j}=\ $};
\node at (-.5,1) [above]{$\scriptstyle{i}$};
\node at (1.5,1) [above]{$\scriptstyle{j}$};
\node at (-2,1) [above]{$\scriptstyle{1}$};
\node at (3,1) [above]{$\scriptstyle{m}$};
\end{tikzpicture}
\ for $1\leq i<j\leq m$.
\end{center}
Let us consider an action of the pure braid group $P_{2k}$ of $2k$ strands on $W_{2k}(n^{\pm})$. 
For any element in $P_{2k}$, 
we depict the braid diagram of it and replace each strand with $n$-parallelized strands. 
The action is obtained by gluing this parallelized braid diagram on the top of $(D,P)$, see Fig.~\ref{Pnrep} in the case of $k=3$. 
Of course, 
the bottom and the top of each strand of any element in $P_{2k}$ has the same index. 
Thus, 
the above operation preserves elements in $W_{2k}(n^{\pm})$. 
We denote this action by $\rho_n\colon P_{2k}\to GL(W_{2k}(n^{\pm}))$. 
The triangle-free basis $\mathcal{B}_T$ is naturally ordered by the lexicographic order of $(j_0,j_1,\dots,j_{k-2})$. 
It derives the matrix representation of $\rho_{n}^{T}\colon P_{2k} \to GL_{(n+1)^{k-2}}(\mathbb{C}(q^{\frac{1}{6}}))$.
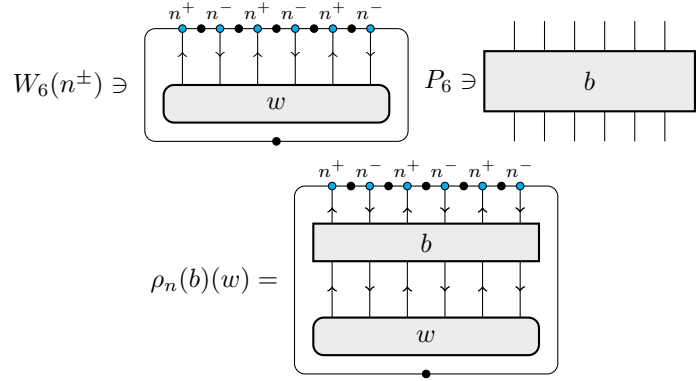
\begin{figure}
\begin{tikzpicture}[baseline=-.6ex, scale=0.5, yshift=3cm]
\draw[rounded corners] (0,-3) rectangle (7,0);
\draw [-<-=.5] (1,0) -- (1,-1.5);
\draw [->-=.5] (2,0) -- (2,-1.5);
\draw [-<-=.5] (3,0) -- (3,-1.5);
\draw [->-=.5] (4,0) -- (4,-1.5);
\draw [-<-=.5] (5,0) -- (5,-1.5);
\draw [->-=.5] (6,0) -- (6,-1.5);
\draw[fill=cyan] (1,0)  circle (.1);
\draw[fill=cyan] (2,0)  circle (.1);
\draw[fill=cyan] (3,0)  circle (.1);
\draw[fill=cyan] (4,0)  circle (.1);
\draw[fill=cyan] (5,0)  circle (.1);
\draw[fill=cyan] (6,0)  circle (.1);
\node at (1,0) [above]{$\scriptstyle{n^{+}}$};
\node at (2,0) [above]{$\scriptstyle{n^{-}}$};
\node at (3,0) [above]{$\scriptstyle{n^{+}}$};
\node at (4,0) [above]{$\scriptstyle{n^{-}}$};
\node at (5,0) [above]{$\scriptstyle{n^{+}}$};
\node at (6,0) [above]{$\scriptstyle{n^{-}}$};
\draw[fill=black] (3.5,-3)  circle (.1);
\draw[fill=black] (1.5,0)  circle (.1);
\draw[fill=black] (2.5,0)  circle (.1);
\draw[fill=black] (3.5,0)  circle (.1);
\draw[fill=black] (4.5,0)  circle (.1);
\draw[fill=black] (5.5,0)  circle (.1);
\draw[thick, rounded corners, fill=lightgray!30] (.5,-2.5) rectangle (6.5,-1.5);
\node at (3.5,-2) {$w$};
\node at (0,-1.5) [left]{$W_{6}(n^{\pm})\ni\ $};
\end{tikzpicture}
\begin{tikzpicture}[baseline=-.6ex, scale=0.4]
\draw (1,0) -- (1,4);
\draw (2,0) -- (2,4);
\draw (3,0) -- (3,4);
\draw (4,0) -- (4,4);
\draw (5,0) -- (5,4);
\draw (6,0) -- (6,4);
\draw[thick, fill=lightgray!30] (0,1) rectangle (7,3);
\node at (3.5,2) {$b$};
\node at (-1,2) {$P_{6}\ni\ $};
\end{tikzpicture}\\
\begin{tikzpicture}[baseline=-.6ex, scale=0.5]
\draw[rounded corners] (0,-5) rectangle (7,0);
\draw [-<-=.2, -<-=.8] (1,0) -- (1,-3.5);
\draw [->-=.2, ->-=.8] (2,0) -- (2,-3.5);
\draw [-<-=.2, -<-=.8] (3,0) -- (3,-3.5);
\draw [->-=.2, ->-=.8] (4,0) -- (4,-3.5);
\draw [-<-=.2, -<-=.8] (5,0) -- (5,-3.5);
\draw [->-=.2, ->-=.8] (6,0) -- (6,-3.5);
\draw[fill=cyan] (1,0)  circle (.1);
\draw[fill=cyan] (2,0)  circle (.1);
\draw[fill=cyan] (3,0)  circle (.1);
\draw[fill=cyan] (4,0)  circle (.1);
\draw[fill=cyan] (5,0)  circle (.1);
\draw[fill=cyan] (6,0)  circle (.1);
\node at (1,0) [above]{$\scriptstyle{n^{+}}$};
\node at (2,0) [above]{$\scriptstyle{n^{-}}$};
\node at (3,0) [above]{$\scriptstyle{n^{+}}$};
\node at (4,0) [above]{$\scriptstyle{n^{-}}$};
\node at (5,0) [above]{$\scriptstyle{n^{+}}$};
\node at (6,0) [above]{$\scriptstyle{n^{-}}$};
\draw[fill=black] (3.5,-5)  circle (.1);
\draw[fill=black] (1.5,0)  circle (.1);
\draw[fill=black] (2.5,0)  circle (.1);
\draw[fill=black] (3.5,0)  circle (.1);
\draw[fill=black] (4.5,0)  circle (.1);
\draw[fill=black] (5.5,0)  circle (.1);
\draw[thick, fill=lightgray!30] (.5,-2) rectangle (6.5,-1);
\node at (3.5,-1.5) {$b$};
\draw[thick, rounded corners, fill=lightgray!30] (.5,-4.5) rectangle (6.5,-3.5);
\node at (3.5,-4) {$w$};
\node at (0,-2.5) [left]{$\rho_n(b)(w)=\ $};
\end{tikzpicture}
\caption{Representation of $P_{6}$}
\label{Pnrep}
\end{figure}
For example, 
we can easily compute the matrix $\rho_{n}^{T}(\sigma_1^2)$. 
Let $T$ be a colored triangle-free triangulation of $(D,Q_{2k})$ such that it has a short chord from $q_0$ to $q_2$ colored with $0$. 
Then, 
$\rho_{n}^{T}(\sigma_1^2)$ can be computed by using the following formula.
\begin{LEM}[{\cite[Lemma~3.7]{Yuasa17b}}]\label{twist}
For any $j=0,1,\dots, n$, 
\begin{itemize}
\item
$
\bigg\langle\,\tikz[baseline=-.6ex]{
\draw[triple={[line width=1.4pt, white] in [line width=2.2pt, black] in [line width=5.4pt, white]}]
(-1.5,0) -- (-1,0);
\draw (-1,0) to[out=north east, in=west] (-.5,.4);
\draw (-1,0) to[out=south east, in=west] (-.5,-.4);
\draw[->-=1, white, double=black, double distance=0.4pt, ultra thick] 
(-.5,.4) to[out=east, in=west] (.0,-.4);
\draw[-<-=1, white, double=black, double distance=0.4pt, ultra thick] 
(-.5,-.4) to[out=east, in=west] (.0,.4);
\draw[white, double=black, double distance=0.4pt, ultra thick] 
(0,.4) to[out=east, in=west] (.5,-.4);
\draw[white, double=black, double distance=0.4pt, ultra thick] 
(0,-.4) to[out=east, in=west] (.5,.4);
\draw[fill=white] (-1,0) circle (.1);
\node at (-.5,-.5) [left]{$\scriptstyle{n}$};
\node at (-.5,.5) [left]{$\scriptstyle{n}$};
\node at (-1.5,0) [above]{$\scriptstyle{j}$};
}\,\bigg\rangle_{\!3}
=
q^{\frac{2}{3}(n^2+3n)-j^2-2j}
\bigg\langle\,\tikz[baseline=-.6ex]{
\draw[triple={[line width=1.4pt, white] in [line width=2.2pt, black] in [line width=5.4pt, white]}]
(-1.5,0) -- (-1,0);
\draw[->-=.5] (-1,0) to[out=north east, in=west] (-.5,.4);
\draw[-<-=.5] (-1,0) to[out=south east, in=west] (-.5,-.4);
\draw[fill=white] (-1,0) circle (.1);
\node at (-.5,-.5) [left]{$\scriptstyle{n}$};
\node at (-.5,.5) [left]{$\scriptstyle{n}$};
\node at (-1.5,0) [above]{$\scriptstyle{j}$};
}\,\bigg\rangle_{\!3}$, 
\item
$\bigg\langle\,\tikz[baseline=-.6ex]{
\draw[triple={[line width=1.4pt, white] in [line width=2.2pt, black] in [line width=5.4pt, white]}]
(-1.5,0) -- (-1,0);
\draw (-1,0) to[out=north east, in=west] (-.5,.4);
\draw (-1,0) to[out=south east, in=west] (-.5,-.4);
\draw[-<-=1, white, double=black, double distance=0.4pt, ultra thick] 
(-.5,-.4) to[out=east, in=west] (.0,.4);
\draw[->-=1, white, double=black, double distance=0.4pt, ultra thick] 
(-.5,.4) to[out=east, in=west] (.0,-.4);
\draw[white, double=black, double distance=0.4pt, ultra thick] 
(0,-.4) to[out=east, in=west] (.5,.4);
\draw[white, double=black, double distance=0.4pt, ultra thick] 
(0,.4) to[out=east, in=west] (.5,-.4);
\draw[fill=white] (-1,0) circle (.1);
\node at (-.5,-.5) [left]{$\scriptstyle{n}$};
\node at (-.5,.5) [left]{$\scriptstyle{n}$};
\node at (-1.5,0) [above]{$\scriptstyle{j}$};
}\,\bigg\rangle_{\!3}
=
q^{-\frac{2}{3}(n^2+3n)+j^2+2j}
\bigg\langle\,\tikz[baseline=-.6ex]{
\draw[triple={[line width=1.4pt, white] in [line width=2.2pt, black] in [line width=5.4pt, white]}]
(-1.5,0) -- (-1,0);
\draw[->-=.5] (-1,0) to[out=north east, in=west] (-.5,.4);
\draw[-<-=.5] (-1,0) to[out=south east, in=west] (-.5,-.4);
\draw[fill=white] (-1,0) circle (.1);
\node at (-.5,-.5) [left]{$\scriptstyle{n}$};
\node at (-.5,.5) [left]{$\scriptstyle{n}$};
\node at (-1.5,0) [above]{$\scriptstyle{j}$};
}\,\bigg\rangle_{\!3}$\, .
\end{itemize}
\end{LEM}

Consequently, 
we obtain $\rho_{n}(\beta_T(j_0,j_1,\dots,j_{k-2}))=q^{-\frac{2}{3}(n^2+3n)+j_0^2+2j_0}\beta_T(j_0,j_1,\dots,j_{k-2})$

The following proposition describes a relationship between $\rho_n^{T}$ and $\rho_n^{T'}$. 
\begin{PROP}
For any colored triangle-free triangulations $T$ and $T'$,
There exists $A\in GL_{(n+1)^{k-2}}(\mathbb{C}(q^{\frac{1}{6}}))$ such that 
$\rho_n^{T}=A^{-1}\rho_n^{T'}A$.
\end{PROP}
\begin{proof}
Each chord is contained a unique quadrangle as its diagonal. 
A {\em flip} is replacing the chord with the other diagonal. 
A flip of a colored triangle-free triangulation is a flip between triangle-free triangulations such that the flipped chord has the same color of the previous one.
Adin, Firer and Roichman~\cite[Proposition~3.2]{AdinFirerRoichman10} defined a transitive action of the affine Weyl group $\tilde{C}_{k}$ on the colored triangle-free triangulations of $(D,Q_k)$. 
A generator $s_i$ of $\tilde{C}_{k}$ acts by a flip of the chord colored with $i$. 
Therefore, 
the colored triangle-free triangulation $T$ transform into $T'$ by a sequence of flips of colored triangle-free triangulations. 
A flip of colored triangle-free triangulations induces the change of triangle-free basis of $W_{2k}(n^{\pm})$.  
\end{proof}

\subsection{Computation of the Representation $\rho_n^{T}$}
Let us compute the value of $\rho_{n}^{T}$ on the generators $A_{i,j}$ of $P_n$. 
At the beginning, 
we prepare a lemma and a colored triangle-free triangulation $T_0$ fitting for the calculation of $A_{i,j}$.
\begin{LEM}\label{twitwist} For any $j=0,1,\dots,n$,
\begin{itemize}
\item
$
\bigg\langle\,\tikz[baseline=-.6ex]{
\draw[-<-=.5, white, double=black, double distance=0.4pt, ultra thick] 
(-1.5,0) -- (-1,0);
\draw[triple={[line width=1.4pt, white] in [line width=2.2pt, black] in [line width=5.4pt, white]}]
(-.5,.4) to[out=east, in=west] (.0,-.4);
\draw[triple={[line width=1.4pt, white] in [line width=2.2pt, black] in [line width=5.4pt, white]}]
(-1,0) to[out=north east, in=west] (-.5,.4);
\draw (-1,0) to[out=south east, in=west] (-.5,-.4);
\draw[-<-=1, white, double=black, double distance=0.4pt, ultra thick] 
(-.5,-.4) to[out=east, in=west] (.0,.4);
\draw[white, double=black, double distance=0.4pt, ultra thick] 
(0,.4) to[out=east, in=west] (.5,-.4);
\draw[triple={[line width=1.4pt, white] in [line width=2.2pt, black] in [line width=5.4pt, white]}] 
(0,-.4) to[out=east, in=west] (.5,.4);
\draw[fill=white] (-1,0) circle (.1);
\node at (-.5,-.5) [left]{$\scriptstyle{n}$};
\node at (-.5,.5) [left]{$\scriptstyle{j}$};
\node at (-1.5,0) [above]{$\scriptstyle{n}$};
}\,\bigg\rangle_{\!3}
=
q^{j^2+2j}
\bigg\langle\,\tikz[baseline=-.6ex]{
\draw[-<-=.5] (-1.5,0) -- (-1,0);
\draw[triple={[line width=1.4pt, white] in [line width=2.2pt, black] in [line width=5.4pt, white]}] 
(-1,0) to[out=north east, in=west] (-.5,.4);
\draw[-<-=.5] (-1,0) to[out=south east, in=west] (-.5,-.4);
\draw[fill=white] (-1,0) circle (.1);
\node at (-.5,-.5) [left]{$\scriptstyle{n}$};
\node at (-.5,.5) [left]{$\scriptstyle{j}$};
\node at (-1.5,0) [above]{$\scriptstyle{n}$};
}\,\bigg\rangle_{\!3}\, $,
\item
$\bigg\langle\,\tikz[baseline=-.6ex]{
\draw[-<-=.5, white, double=black, double distance=0.4pt, ultra thick] 
(-1.5,0) -- (-1,0);
\draw[triple={[line width=1.4pt, white] in [line width=2.2pt, black] in [line width=5.4pt, white]}]
(-1,0) to[out=north east, in=west] (-.5,.4);
\draw (-1,0) to[out=south east, in=west] (-.5,-.4);
\draw[-<-=1, white, double=black, double distance=0.4pt, ultra thick] 
(-.5,-.4) to[out=east, in=west] (.0,.4);
\draw[triple={[line width=1.4pt, white] in [line width=2.2pt, black] in [line width=5.4pt, white]}]
(-.5,.4) to[out=east, in=west] (.0,-.4);
\draw[triple={[line width=1.4pt, white] in [line width=2.2pt, black] in [line width=5.4pt, white]}] 
(0,-.4) to[out=east, in=west] (.5,.4);
\draw[white, double=black, double distance=0.4pt, ultra thick] 
(0,.4) to[out=east, in=west] (.5,-.4);
\draw[fill=white] (-1,0) circle (.1);
\node at (-.5,-.5) [left]{$\scriptstyle{n}$};
\node at (-.5,.5) [left]{$\scriptstyle{j}$};
\node at (-1.5,0) [above]{$\scriptstyle{n}$};
}\,\bigg\rangle_{\!3}
=
q^{-j^2-2j}
\bigg\langle\,\tikz[baseline=-.6ex]{
\draw[-<-=.5] (-1.5,0) -- (-1,0);
\draw[triple={[line width=1.4pt, white] in [line width=2.2pt, black] in [line width=5.4pt, white]}] 
(-1,0) to[out=north east, in=west] (-.5,.4);
\draw[-<-=.5] (-1,0) to[out=south east, in=west] (-.5,-.4);
\draw[fill=white] (-1,0) circle (.1);
\node at (-.5,-.5) [left]{$\scriptstyle{n}$};
\node at (-.5,.5) [left]{$\scriptstyle{j}$};
\node at (-1.5,0) [above]{$\scriptstyle{n}$};
}\,\bigg\rangle_{\!3}\,$ .
\end{itemize}
\end{LEM}
\begin{proof}
From Lemma~\ref{twist},
\[
\bigg\langle\,\tikz[baseline=-.6ex, xscale=-1]{
\draw[white, double=black, double distance=0.4pt, ultra thick] 
(-1.5,.4) to[out=west, in=north east] (-2,-.4);
\draw[triple={[line width=1.4pt, white] in [line width=2.2pt, black] in [line width=5.4pt, white]}]
(-2,.4) to[out=south east, in=west] (-1.5,-.4);
\draw[triple={[line width=1.4pt, white] in [line width=2.2pt, black] in [line width=5.4pt, white]}]
(-1.5,-.4) to[out=east, in=north west] (-1,0);
\draw (-1,0) to[out=north east, in=west] (-.5,.4);
\draw (-1,0) to[out=south east, in=west] (-.5,-.4);
\draw[-<-=1, white, double=black, double distance=0.4pt, ultra thick] 
(-.5,-.4) to[out=east, in=west] (.0,.4);
\draw[->-=1, white, double=black, double distance=0.4pt, ultra thick] 
(-.5,.4) to[out=east, in=west] (.0,-.4);
\draw[white, double=black, double distance=0.4pt, ultra thick] 
(0,-.4) to[out=east, in=west] (.5,.4);
\draw[white, double=black, double distance=0.4pt, ultra thick] 
(0,.4) to[out=east, in=west] (.5,-.4);
\draw[white, double=black, double distance=0.4pt, ultra thick] 
(.5,-.4) to[out=east, in=east] (.5,-.6) -- (-1,-.6);
\draw[white, double=black, double distance=0.4pt, ultra thick] 
(-1,-.6) to[out=west, in=east] (-1.5,.4);
\draw[fill=white] (-1,0) circle (.1);
\node at (.5,-.4) [left]{$\scriptstyle{n}$};
\node at (.5,.4) [left]{$\scriptstyle{n}$};
\node at (-2,.4) [right]{$\scriptstyle{j}$};
}\,\bigg\rangle_{\!3}
=q^{\frac{2}{3}(n^2+3n)-j^2-2j}
\bigg\langle\,\tikz[baseline=-.6ex]{
\draw[-<-=.5, white, double=black, double distance=0.4pt, ultra thick] 
(-1.5,0) -- (-1,0);
\draw[triple={[line width=1.4pt, white] in [line width=2.2pt, black] in [line width=5.4pt, white]}]
(-.5,.4) to[out=east, in=west] (.0,-.4);
\draw[triple={[line width=1.4pt, white] in [line width=2.2pt, black] in [line width=5.4pt, white]}]
(-1,0) to[out=north east, in=west] (-.5,.4);
\draw (-1,0) to[out=south east, in=west] (-.5,-.4);
\draw[-<-=1, white, double=black, double distance=0.4pt, ultra thick] 
(-.5,-.4) to[out=east, in=west] (.0,.4);
\draw[white, double=black, double distance=0.4pt, ultra thick] 
(0,.4) to[out=east, in=west] (.5,-.4);
\draw[triple={[line width=1.4pt, white] in [line width=2.2pt, black] in [line width=5.4pt, white]}] 
(0,-.4) to[out=east, in=west] (.5,.4);
\draw[fill=white] (-1,0) circle (.1);
\node at (-.5,-.5) [left]{$\scriptstyle{n}$};
\node at (-.5,.5) [left]{$\scriptstyle{j}$};
\node at (-1.5,0) [above]{$\scriptstyle{n}$};
}\,\bigg\rangle_{\!3}.\]
By a regular isotopy,
The left-hand side of the above equation is deformed into 
\[
\bigg\langle\,\tikz[baseline=-.6ex]{
\draw[-<-=.5] (-2,0) -- (-1,0);
\draw[triple={[line width=1.4pt, white] in [line width=2.2pt, black] in [line width=5.4pt, white]}] 
(-1,0) to[out=north east, in=west] (-.5,.5);
\draw[white, double=black, double distance=0.4pt, ultra thick] 
(-1.5,-.2) to[out=east, in=west] (-.5,-.5);
\draw[-<-=.8, white, double=black, double distance=0.4pt, ultra thick] 
(-1,0) to[out=south, in=east] (-1.5,-.6);
\draw[white, double=black, double distance=0.4pt, ultra thick] 
(-1.5,-.6) to[out=west, in=north] (-2,-.4);
\draw[white, double=black, double distance=0.4pt, ultra thick] 
(-2,-.4) to[out=south, in=west] (-1.5,-.2);
\draw[fill=white] (-1,0) circle (.1);
\node at (-.5,-.5) [right]{$\scriptstyle{n}$};
\node at (-.5,.5) [right]{$\scriptstyle{j}$};
\node at (-1.5,0) [above]{$\scriptstyle{n}$};
}\,\bigg\rangle_{\!3}
=q^{\frac{2}{3}(n^2+3n)}
\bigg\langle\,\tikz[baseline=-.6ex]{
\draw[-<-=.5] (-1.5,0) -- (-1,0);
\draw[triple={[line width=1.4pt, white] in [line width=2.2pt, black] in [line width=5.4pt, white]}] 
(-1,0) to[out=north east, in=west] (-.5,.4);
\draw[-<-=.5] (-1,0) to[out=south east, in=west] (-.5,-.4);
\draw[fill=white] (-1,0) circle (.1);
\node at (-.5,-.5) [left]{$\scriptstyle{n}$};
\node at (-.5,.5) [left]{$\scriptstyle{j}$};
\node at (-1.5,0) [above]{$\scriptstyle{n}$};
}\,\bigg\rangle_{\!3}\, .
\]
We can similarly prove the second equation.
\end{proof}
\begin{figure}
\begin{tikzpicture}[baseline=-.6ex, scale=0.5]
\draw[fill=lightgray] (-2,0) -- (2,0) to[out=south, in=south] (-2,0) -- cycle;
\draw[rounded corners] (-9.5,-5) rectangle (9.5,0);
\draw[thick, magenta] (-2,0) to[out=south, in=south] (2,0);
\draw[thick, magenta] (2,0) to[out=south, in=south] (-4,0);
\draw[thick, magenta] (-4,0) to[out=south, in=south] (4,0);
\draw[thick, magenta] (6,0) to[out=south, in=south] (-6,0);
\draw[thick, magenta] (-6,0) to[out=south, in=south] (8,0);
\draw[thick, magenta] (8,0) to[out=south, in=south] (-8,0);
\node[thick, magenta] at (-5,0) [below]{${\cdots}$};
\node[thick, magenta] at (5,0) [below]{${\cdots}$};
\draw[fill=black] (0,-5)  circle (.1);
\draw[fill=black] (-8,0)  circle (.1);
\draw[fill=black] (-6,0)  circle (.1);
\draw[fill=black] (-4,0)  circle (.1);
\draw[fill=black] (-2,0)  circle (.1);
\draw[fill=black] (0,0)  circle (.1);
\draw[fill=black] (2,0)  circle (.1);
\draw[fill=black] (4,0)  circle (.1);
\draw[fill=black] (6,0)  circle (.1);
\draw[fill=black] (8,0)  circle (.1);
\node at (0,-5) [below]{$\scriptstyle{q_0}$};
\node at (-8,0) [above]{$\scriptstyle{q_{1}}$};
\node at (-6,0) [above]{$\scriptstyle{q_{2}}$};
\node at (-5,0) [above]{$\scriptstyle{\cdots}$};
\node at (-4,0) [above]{$\scriptstyle{q_{k-2}}$};
\node at (-2,0) [above]{$\scriptstyle{q_{k-1}}$};
\node at (0,0) [above]{$\scriptstyle{q_k}$};
\node at (2,0) [above]{$\scriptstyle{q_{k+1}}$};
\node at (4,0) [above]{$\scriptstyle{q_{k+2}}$};
\node at (5,0) [above]{$\scriptstyle{\cdots}$};
\node at (6,0) [above]{$\scriptstyle{q_{2k-2}}$};
\node at (8,0) [above]{$\scriptstyle{q_{2k-1}}$};
\node at (0,0) [above]{$\scriptstyle{q_k}$};
\node at (-9.5,-3) [left]{$T_0=\ $};
\end{tikzpicture}\\
\begin{tikzpicture}[baseline=-.6ex, scale=0.5]
\draw[rounded corners] (-9.5,-5) rectangle (9.5,0);
\draw [triple={[line width=1.4pt, white] in [line width=2.2pt, black] in [line width=5.4pt, white]}] (0,-.5) -- (0,-1);
\draw [triple={[line width=1.4pt, white] in [line width=2.2pt, black] in [line width=5.4pt, white]}] (2,-1) -- (2,-1.5);
\draw [triple={[line width=1.4pt, white] in [line width=2.2pt, black] in [line width=5.4pt, white]}] (5,-2) -- (5,-2.5);
\draw [triple={[line width=1.4pt, white] in [line width=2.2pt, black] in [line width=5.4pt, white]}] (6,-2.5) -- (6,-3);
\draw[rounded corners] (-1,0) -- (-1,-.5) -- (0,-.5);
\draw[rounded corners] (-3,0) -- (-3,-1) -- (0,-1);
\draw[rounded corners] (-4,0) -- (-4,-1.5) -- (2,-1.5);
\draw[->-=.1, rounded corners] (-7,0) -- (-7,-2.5) -- (5,-2.5);
\draw[-<-=.1, rounded corners] (-9,0) -- (-9,-3) -- (6,-3);
\draw[rounded corners] (1,0) -- (1,-.5) -- (0,-.5);
\draw[rounded corners] (3,0) -- (3,-1) -- (2,-1);
\draw[-<-=.2, rounded corners] (6,0) -- (6,-2) -- (5,-2);
\draw[->-=.2, rounded corners] (7,0) -- (7,-2.5) -- (6,-2.5);
\draw[-<-=.1, rounded corners] (9,0) -- (9,-3) -- (6,-3);
\draw[rounded corners] (0,-1) -- (2,-1);
\draw[rounded corners] (2,-1.5) -- (3,-1.5);
\draw[rounded corners] (5,-2) -- (4,-2);
\draw[rounded corners] (5,-2.5) -- (6,-2.5);
\draw[fill=white] (0,-.5)  circle (.1);
\draw[fill=white] (0,-1)  circle (.1);
\draw[fill=white] (2,-1)  circle (.1);
\draw[fill=white] (2,-1.5)  circle (.1);
\draw[fill=white] (5,-2)  circle (.1);
\draw[fill=white] (5,-2.5)  circle (.1);
\draw[fill=white] (6,-2.5)  circle (.1);
\draw[fill=white] (6,-3)  circle (.1);
\draw[fill=cyan] (-9,0)  circle (.1);
\draw[fill=cyan] (-7,0)  circle (.1);
\draw[fill=cyan] (-3,0)  circle (.1);
\draw[fill=cyan] (-4,0)  circle (.1);
\draw[fill=cyan] (-1,0)  circle (.1);
\draw[fill=cyan] (1,0)  circle (.1);
\draw[fill=cyan] (3,0)  circle (.1);
\draw[fill=cyan] (6,0)  circle (.1);
\draw[fill=cyan] (7,0)  circle (.1);
\draw[fill=cyan] (9,0)  circle (.1);
\node at (3.5,-1.5) [below]{$\scriptstyle{\cdots}$};
\node at (-9,0) [above]{$\scriptstyle{P^{(1)}}$};
\node at (-7,0) [above]{$\scriptstyle{P^{(2)}}$};
\node at (-5,0) [above]{$\scriptstyle{\cdots}$};
\node at (-3,0) [above]{$\scriptstyle{P^{(k-1)}}$};
\node at (-1,0) [above]{$\scriptstyle{P^{(k)}}$};
\node at (1,0) [above]{$\scriptstyle{P^{(k+1)}}$};
\node at (3,0) [above]{$\scriptstyle{P^{(k+2)}}$};
\node at (5,0) [above]{$\scriptstyle{\cdots}$};
\node at (7,0) [above]{$\scriptstyle{P^{(2k-1)}}$};
\node at (9,0) [above]{$\scriptstyle{P^{(2k)}}$};
\node at (-9.5,-3) [left]{$\beta_{T_0}=\ $};
\end{tikzpicture}
\caption{Colored triangle-free triangulation $T_{0}$ and $\beta_{T_0}$}
\label{Fig:T0}
\end{figure}

Let us define a special colored triangulation $T_0$ of $(D,Q_{2k})$ as in Fig.~\ref{Fig:T0}. 
The shaded triangle contains the short chord colored by $0$. 
In this case, 
we can compute $\rho_n^{T_0}(A_{k-t,k+t})$ and $\rho_n^{T_0}(A_{k-t,k+t+1})$ for $1\leq t\leq k-1$. 
We remark that $\rho_n^{T_0}(A_{k,k+1})$ is already computed in the previous section. 

First, 
we calculate $\rho_n^{T_0}(A_{k-t,k+t})$, that is, 
we consider the clasped $A_2$ web obtained by gluing $A_{k-t,k+t}$ on the top of $\beta_{T_0}$. 
We can slide a subarc of the $(k+t)$-th strand knotted with the $(k-t)$-th strand downward. 
Then, 
the difference of the $A_2$ web and $\beta_n^{T_0}$ only appears around the clasped $A_2$ web of type $(j_{t},j_{t})$ and one of type $(j_{t-1},j_{t-1})$. 
This part is calculated by using Lemma~\ref{twist} and Lemma~\ref{twitwist} as follows:
\begin{align*}
&\Bigg\langle\,\tikz[baseline=-.6ex, scale=.5, yshift=-1cm]{
\draw[white, double=black, double distance=0.4pt, ultra thick, rounded corners] 
(-.5,.5) -- (-.5,-.5) -- (-1,-.5);
\draw[white, double=black, double distance=0.4pt, ultra thick, rounded corners] 
(-3,0) -- (1.5,0);
\draw[white, double=black, double distance=0.4pt, ultra thick, rounded corners] 
(-3,1) -- (1,1) -- (1,1.5) -- (0,1.5);
\draw[white, double=black, double distance=0.4pt, ultra thick, rounded corners] 
(-1,-.5) -- (-1.5,-.5) -- (-1.5,2) -- (1,2) -- (1,2.5);
\draw[white, double=black, double distance=0.4pt, ultra thick, rounded corners] 
(0,1.5) -- (-.5,1.5) -- (-.5,.5);
\draw[triple={[line width=1.4pt, white] in [line width=2.2pt, black] in [line width=5.4pt, white]}]
(0,0) -- (0,1);
\draw[triple={[line width=1.4pt, white] in [line width=2.2pt, black] in [line width=5.4pt, white]}]
(-2,1) -- (-2,2);
\draw[fill=white] (0,0) circle (.1);
\draw[fill=white] (0,1) circle (.1);
\draw[fill=white] (-2,1) circle (.1);
\node at (-2,1.5) [left]{$\scriptstyle{j_{t}}$};
\node at (0,.5) [right]{$\scriptstyle{j_{t-1}}$};
}\,\Bigg\rangle_{\!3}
=
q^{\frac{2}{3}(n^2+3n)-j_{t-1}^2-2j_{t-1}}
\Bigg\langle\,\tikz[baseline=-.6ex, scale=.5, yshift=-1cm]{
\draw[white, double=black, double distance=0.4pt, ultra thick, rounded corners] 
(0,1.5) -- (-.5,1.5) -- (-.5,-.5) -- (-1,-.5);
\draw[white, double=black, double distance=0.4pt, ultra thick, rounded corners] 
(-3,0) -- (1.5,0);
\draw[white, double=black, double distance=0.4pt, ultra thick, rounded corners] 
(-3,1) -- (1,1) -- (1,1.5) -- (0,1.5);
\draw[white, double=black, double distance=0.4pt, ultra thick, rounded corners] 
(-1,-.5) -- (-1.5,-.5) -- (-1.5,2) -- (1,2) -- (1,2.5);
\draw[triple={[line width=1.4pt, white] in [line width=2.2pt, black] in [line width=5.4pt, white]}]
(0,0) -- (0,1);
\draw[triple={[line width=1.4pt, white] in [line width=2.2pt, black] in [line width=5.4pt, white]}]
(-2,1) -- (-2,2);
\draw[fill=white] (0,0) circle (.1);
\draw[fill=white] (0,1) circle (.1);
\draw[fill=white] (-2,1) circle (.1);
\node at (-2,1.5) [left]{$\scriptstyle{j_{t}}$};
\node at (0,.5) [right]{$\scriptstyle{j_{t-1}}$};
}\,\Bigg\rangle_{\!3}\\
&=\sum_{a=0}^{n}
q^{\frac{2}{3}(n^2+3n)-j_{t-1}^2-2j_{t-1}}
\begin{Bmatrix}
n^{-}&n^{+}&(a,a)\\
n^{-}&n^{+}&(j_{t-1},j_{t-1})
\end{Bmatrix}
\Bigg\langle\,\tikz[baseline=-.6ex, scale=.5, yshift=-1cm]{
\draw[white, double=black, double distance=0.4pt, ultra thick, rounded corners] 
(.5,1.5) -- (0,1.5) -- (0,-.5) -- (-.2,-.5);
\draw[white, double=black, double distance=0.4pt, ultra thick, rounded corners] 
(-3,0) -- (-1.5,0) -- (-1,.5);
\draw[white, double=black, double distance=0.4pt, ultra thick, rounded corners] 
(-3,1) -- (-1.5,1) -- (-1,.5);
\draw[white, double=black, double distance=0.4pt, ultra thick, rounded corners] 
(.5,.5) -- (1,0) -- (1.5,0);
\draw[white, double=black, double distance=0.4pt, ultra thick, rounded corners] 
(.5,.5) -- (1,1) -- (1,1.5) -- (.5,1.5);
\draw[triple={[line width=1.4pt, white] in [line width=2.2pt, black] in [line width=5.4pt, white]}]
(-1,.5) -- (.5,.5);
\draw[triple={[line width=1.4pt, white] in [line width=2.2pt, black] in [line width=5.4pt, white]}]
(-2,1) -- (-2,2);
\draw[white, double=black, double distance=0.4pt, ultra thick, rounded corners] 
(-.2,-.5) -- (-.5,-.5) -- (-.5,2) -- (1,2) -- (1,2.5);
\draw[fill=white] (-1,.5) circle (.1);
\draw[fill=white] (.5,.5) circle (.1);
\draw[fill=white] (-2,1) circle (.1);
\node at (-2,1.5) [left]{$\scriptstyle{j_{t}}$};
\node at (.5,.5) [right]{$\scriptstyle{a}$};
}\,\Bigg\rangle_{\!3}\\
&=\sum_{a=0}^{n}
q^{\frac{2}{3}(n^2+3n)-j_{t-1}^2-2j_{t-1}}
q^{-a^2-2a}
\begin{Bmatrix}
n^{-}&n^{+}&(a,a)\\
n^{-}&n^{+}&(j_{t-1},j_{t-1})
\end{Bmatrix}
\Bigg\langle\,\tikz[baseline=-.6ex, scale=.5, yshift=-1cm]{
\draw[white, double=black, double distance=0.4pt, ultra thick, rounded corners] 
(-3,0) -- (-1.5,0) -- (-1,.5);
\draw[white, double=black, double distance=0.4pt, ultra thick, rounded corners] 
(-3,1) -- (-1.5,1) -- (-1,.5);
\draw[white, double=black, double distance=0.4pt, ultra thick, rounded corners] 
(.5,.5) -- (1,0) -- (1.5,0);
\draw[white, double=black, double distance=0.4pt, ultra thick, rounded corners] 
(.5,.5) -- (1,1) -- (1,2);
\draw[triple={[line width=1.4pt, white] in [line width=2.2pt, black] in [line width=5.4pt, white]}]
(-1,.5) -- (.5,.5);
\draw[triple={[line width=1.4pt, white] in [line width=2.2pt, black] in [line width=5.4pt, white]}]
(-2,1) -- (-2,2);
\draw[fill=white] (-1,.5) circle (.1);
\draw[fill=white] (.5,.5) circle (.1);
\draw[fill=white] (-2,1) circle (.1);
\node at (-2,1.5) [left]{$\scriptstyle{j_{t}}$};
\node at (0,.5) [above left]{$\scriptstyle{a}$};
}\,\Bigg\rangle_{\!3}\\
&=\sum_{a=0}^{n}\sum_{b=0}^{n}
q^{\frac{2}{3}(n^2+3n)-j_{t-1}^2-2j_{t-1}}
q^{-a^2-2a}
\begin{Bmatrix}
n^{-}&n^{+}&(a,a)\\
n^{-}&n^{+}&(j_{t-1},j_{t-1})
\end{Bmatrix}
\begin{Bmatrix}
n^{-}&n^{+}&(b,b)\\
n^{-}&n^{+}&(a,a)
\end{Bmatrix}\\
&\quad\quad\times\Bigg\langle\,\tikz[baseline=-.6ex, scale=.5, yshift=-1cm]{
\draw[white, double=black, double distance=0.4pt, ultra thick, rounded corners] 
(-3,0) -- (1.5,0);
\draw[white, double=black, double distance=0.4pt, ultra thick, rounded corners] 
(-3,1) -- (1,1) -- (1,2);
\draw[triple={[line width=1.4pt, white] in [line width=2.2pt, black] in [line width=5.4pt, white]}]
(0,0) -- (0,1);
\draw[triple={[line width=1.4pt, white] in [line width=2.2pt, black] in [line width=5.4pt, white]}]
(-2,1) -- (-2,2);
\draw[fill=white] (0,0) circle (.1);
\draw[fill=white] (0,1) circle (.1);
\draw[fill=white] (-2,1) circle (.1);
\node at (-2,1.5) [left]{$\scriptstyle{j_{t}}$};
\node at (0,.5) [right]{$\scriptstyle{b}$};
}\,\Bigg\rangle_{\!3}\ .
\end{align*}
The above calculation is independent of the choice of orientations of $n$-colored $A_2$ webs. 

Next, 
we calculate $\rho_n^{T_0}(A_{k-t,k+t+1})$ in a similar way.
In this case, 
we only have to calculate the following.
\begin{align*}
&\Bigg\langle\,\tikz[baseline=-.6ex, scale=.5, yshift=-1cm]{
\draw[triple={[line width=1.4pt, white] in [line width=2.2pt, black] in [line width=5.4pt, white]}]
(-1,1) -- (-1,2.5);
\draw[white, double=black, double distance=0.4pt, ultra thick, rounded corners] 
(0,1.5) -- (-1.5,1.5) -- (-1.5,.7);
\draw[white, double=black, double distance=0.4pt, ultra thick, rounded corners] 
(-1.5,.7) -- (-1.5,.5) -- (-2,.5);
\draw[white, double=black, double distance=0.4pt, ultra thick, rounded corners] 
(-3,0) -- (1.5,0);
\draw[white, double=black, double distance=0.4pt, ultra thick, rounded corners] 
(-3,1) -- (1,1) -- (1,1.5) -- (0,1.5);
\draw[white, double=black, double distance=0.4pt, ultra thick, rounded corners] 
(-2,.5) -- (-2.5,.5) -- (-2.5,2) -- (1,2) -- (1,2.5);
\draw[triple={[line width=1.4pt, white] in [line width=2.2pt, black] in [line width=5.4pt, white]}]
(0,0) -- (0,1);
\draw[fill=white] (0,0) circle (.1);
\draw[fill=white] (0,1) circle (.1);
\draw[fill=white] (-1,1) circle (.1);
\node at (-1,2.5) [left]{$\scriptstyle{j_{t-1}}$};
\node at (0,.5) [right]{$\scriptstyle{j_{t}}$};
}\,\Bigg\rangle_{\!3}
=
\begin{Bmatrix}
n^{+}&j_{t-1}^{\pm}&n\\
n^{-}&j_{t}^{\mp}&n
\end{Bmatrix}
\Bigg\langle\,\tikz[baseline=-.6ex, scale=.5, yshift=-1cm]{
\draw[triple={[line width=1.4pt, white] in [line width=2.2pt, black] in [line width=5.4pt, white]}]
(0,1) -- (0,2.5);
\draw[white, double=black, double distance=0.4pt, ultra thick, rounded corners] 
(0,1.5) -- (-1.5,1.5) -- (-1.5,.7);
\draw[white, double=black, double distance=0.4pt, ultra thick, rounded corners] 
(-1.5,.7) -- (-1.5,.5) -- (-2,.5);
\draw[white, double=black, double distance=0.4pt, ultra thick, rounded corners] 
(-3,0) -- (1.5,0);
\draw[white, double=black, double distance=0.4pt, ultra thick, rounded corners] 
(-3,1) -- (1,1) -- (1,1.5) -- (0,1.5);
\draw[white, double=black, double distance=0.4pt, ultra thick, rounded corners] 
(-2,.5) -- (-2.5,.5) -- (-2.5,2) -- (1,2) -- (1,2.5);
\draw[triple={[line width=1.4pt, white] in [line width=2.2pt, black] in [line width=5.4pt, white]}]
(-1,0) -- (-1,1);
\draw[fill=white] (-1,0) circle (.1);
\draw[fill=white] (-1,1) circle (.1);
\draw[fill=white] (0,1) circle (.1);
\node at (0,2.5) [left]{$\scriptstyle{j_{t-1}}$};
\node at (-1,.5) [right]{$\scriptstyle{j_{t}}$};
}\,\Bigg\rangle_{\!3}\\
&=
q^{j_{t-1}^2+2j_{t-1}}
\begin{Bmatrix}
n^{+}&j_{t-1}^{\pm}&n\\
n^{-}&j_{t}^{\mp}&n
\end{Bmatrix}
\Bigg\langle\,\tikz[baseline=-.6ex, scale=.5, yshift=-1cm]{
\draw[white, double=black, double distance=0.4pt, ultra thick, rounded corners] 
(.5,1) -- (1,1) -- (1,1.5) -- (-1,1.5);
\draw[triple={[line width=1.4pt, white] in [line width=2.2pt, black] in [line width=5.4pt, white]}]
(0,1) -- (0,2.5);
\draw[white, double=black, double distance=0.4pt, ultra thick, rounded corners] 
(-1,1.5) -- (-1.5,1.5) -- (-1.5,.7);
\draw[white, double=black, double distance=0.4pt, ultra thick, rounded corners] 
(-1.5,.7) -- (-1.5,.5) -- (-2,.5);
\draw[white, double=black, double distance=0.4pt, ultra thick, rounded corners] 
(-3,0) -- (1.5,0);
\draw[white, double=black, double distance=0.4pt, ultra thick, rounded corners] 
(-3,1) -- (.5,1);
\draw[white, double=black, double distance=0.4pt, ultra thick, rounded corners] 
(-2,.5) -- (-2.5,.5) -- (-2.5,2) -- (1,2) -- (1,2.5);
\draw[triple={[line width=1.4pt, white] in [line width=2.2pt, black] in [line width=5.4pt, white]}]
(-1,0) -- (-1,1);
\draw[fill=white] (-1,0) circle (.1);
\draw[fill=white] (-1,1) circle (.1);
\draw[fill=white] (0,1) circle (.1);
\node at (0,2.5) [left]{$\scriptstyle{j_{t-1}}$};
\node at (-1,.5) [right]{$\scriptstyle{j_{t}}$};
}\,\Bigg\rangle_{\!3}\\
&=
q^{-\frac{2}{3}(n^2+3n)+j_{t}^2+2j_{t}}q^{j_{t-1}^2+2j_{t-1}}
\Bigg\langle\,\tikz[baseline=-.6ex, scale=.5, yshift=-1cm]{
\draw[white, double=black, double distance=0.4pt, ultra thick, rounded corners] 
(-3,0) -- (1.5,0);
\draw[white, double=black, double distance=0.4pt, ultra thick, rounded corners] 
(-3,1) -- (1,1) -- (1,2);
\draw[triple={[line width=1.4pt, white] in [line width=2.2pt, black] in [line width=5.4pt, white]}]
(0,0) -- (0,1);
\draw[triple={[line width=1.4pt, white] in [line width=2.2pt, black] in [line width=5.4pt, white]}]
(-2,1) -- (-2,2);
\draw[fill=white] (0,0) circle (.1);
\draw[fill=white] (0,1) circle (.1);
\draw[fill=white] (-2,1) circle (.1);
\node at (-2,1.5) [left]{$\scriptstyle{j_{t-1}}$};
\node at (0,.5) [right]{$\scriptstyle{j_t}$};
}\,\Bigg\rangle_{\!3}\ .
\end{align*}

\begin{THM}\label{AT0}
Let $T_0$ be the colored triangle-free triangulation in Fig.~\ref{Fig:T0}. 
Then, 
\begin{align*}
&\rho_{n}(A_{k-t,k+t})(\beta_{T_0}(j_0,j_1,\dots,j_{k-2}))\\
&=\sum_{j'_t=0}^{n}\sum_{a=0}^{n}
q^{\frac{2}{3}(n^2+3n)-j_{t-1}^2-2j_{t-1}}
q^{-a^2-2a}
\begin{Bmatrix}
n^{-}&n^{+}&(a,a)\\
n^{-}&n^{+}&(j_{t-1},j_{t-1})
\end{Bmatrix}
\begin{Bmatrix}
n^{-}&n^{+}&(j'_t,j'_t)\\
n^{-}&n^{+}&(a,a)
\end{Bmatrix}\\
&\quad\quad\times\beta_{T_0}(j_0,\dots,j_{t-1},j'_t,j_{t+1},\dots,j_{k-2})
\\
&\rho_{n}(A_{k-t,k+t-1})(\beta_{T_0}(j_0,j_1,\dots,j_{k-2}))=q^{-\frac{2}{3}(n^2+3n)+j_{t}^2+2j_{t}}q^{j_{t-1}^2+2j_{t-1}}
\beta_{T_0}(j_0,j_1,\dots,j_{k-2})
\end{align*}
\end{THM}

Finally, 
we define a certain deformation of colored triangle-free triangulation $T_0$. 
Let $s_i$ be a left action on the colored triangle-free triangulations of $(D,Q_{2k})$ which flips $i$-th chord for $i=0,1,\dots,k-2$. 
It is easy to confirm that $s^m$ act on $T_0$ as a sequence of flips of triangle-free triangulations where $s=s_{2k}\dots s_2s_0s_{2k-1}\dots s_3s_1$. 
Moreover, 
$s^mT_0$ is a colored triangle-free triangulation obtained by shifting each chord $q_lq_{l'}$ to $q_{l+m}q_{l'+m}$ where the indices are modulo $2k$.
We denote it by $T_m=s^mT_0$. 
A colored triangle-free triangulation determines a triangle-free basis of $W_{2k}(n^\pm)$.
We take two colored triangle-free triangulation $T$ and $T'$ such that $T'=s_iT$.
Then, for $i'=0,1,\dots,k$, 
\[
\beta_{T}(j_0,j_1,\dots,j_{k-2})
=\sum_{j'_i=0}^{n}
\begin{Bmatrix}
n^{-}&n^{+}&(j'_{i'},j'_{i'})\\
n^{-}&n^{+}&(j_{i'},j_{i'})
\end{Bmatrix}
\beta_{T'}(j_0,\dots,j_{i'}-1,j'_{i'},j_{{i'}+1},\dots,j_{k-2}),
\]
where $i=2i'$ and
\[
\beta_{T}(j_0,j_1,\dots,j_{k-2})
=\begin{Bmatrix}
n^{+}&j_{i'}^{\pm}&n\\
n^{-}&j_{{i'}+1}^{\mp}&n
\end{Bmatrix}
\beta_{T'}(j_0,j_1,\dots,j_{k-2}),
\]
where $i=2i'+1$.
It defines the matrix $S_i$ such that $\mathcal{B}_{T}=\mathcal{B}_{T'}S_i$.
Let us denote the matrix representation of $s$ by $S=S_{2k}\dots S_2S_0S_{2k-1}\dots S_3S_1$.
\begin{LEM}\label{shift}
For any $i$ and $j$ such that $0\leq i<j\leq 2k$, 
\begin{align*}
\rho_n^{T_0}(A_{i,j})&=S^{-(t'-k)}\rho_n^{T_0}(A_{k-t,k+t})S^{t'-k}\quad\text{if $i+j$ is even, }\\
\rho_n^{T_0}(A_{i,j})&=S^{-(t'-k)}\rho_n^{T_0}(A_{k-t,k+t+1})S^{t'-k}\quad\text{if $i+j$ is odd.}
\end{align*}
where $t=\lfloor (j-i)/2\rfloor$ and $t'=\lfloor (i+j)/2\rfloor$.
\end{LEM}
\begin{proof}
The action $s^m$ on $T_0$ shifts the chords of it $m$ times clockwise. 
We denoted it by $T_m$.
A clasped $A_2$ web $\rho_n(A_{i,j})\beta_{T_{t'-k}}(j_0,j_1,\dots,j_{k-2})$ coincides with $\rho_n(A_{k-t,k+t})\beta_{T_0}(j_0,j_1,\dots,j_{k-2})$ or $\rho_n(A_{k-t,k+t})\beta_{T_0}(j_0,j_1,\dots,j_{k-2})$ by rotation. 
Therefor, $\rho_n^{T_{t'-k}}(A_{i,j})=\rho_n^{T_0}(A_{k-t,k+t})$ if $i+j$ is even and $\rho_n^{T_{t'-k}}(A_{i,j})=\rho_n^{T_0}(A_{k-t,k+t+1})$ if $i+j$ is odd.
As a result, 
\begin{align*}
\rho_n(A_{i,j})\mathcal{B}_{T_0}
&=\rho_n(A_{i,j})\mathcal{B}_{T_{t'-k}}S^{t'-k}\\
&=\mathcal{B}_{T_{t'-k}}\rho_n^{T_{t'-k}}(A_{i,j})S^{t'-k}\\
&=\mathcal{B}_{T_{t'-k}}\rho_n^{T_0}(A_{k-t,k+t})S^{t'-k}\\
&=\mathcal{B}_{T_0}S^{-(t'-k)}\rho_n^{T_0}(A_{k-t,k+t})S^{t'-k}\\
\end{align*}
if $i+j$ is even. 
By the same way, 
\[
\rho_n(A_{i,j})\mathcal{B}_{T_0}=\mathcal{B}_{T_0}S^{-(t'-k)}\rho_n^{T_0}(A_{k-t,k+t+1})S^{t'-k}
\] if $i+j$ is odd.
\end{proof}

Consequently, 
we can determine the matrix representation with respect to $T_0$ for any generators by Theorem~\ref{AT0} and Lemma~\ref{shift} for any generators..
\bibliographystyle{amsalpha}
\bibliography{PureBraidA2rep}

\def\cprime{$'$}
\providecommand{\bysame}{\leavevmode\hbox to3em{\hrulefill}\thinspace}
\providecommand{\MR}{\relax\ifhmode\unskip\space\fi MR }
% \MRhref is called by the amsart/book/proc definition of \MR.
\providecommand{\MRhref}[2]{%
  \href{http://www.ams.org/mathscinet-getitem?mr=#1}{#2}
}
\providecommand{\href}[2]{#2}
\begin{thebibliography}{BHMV95}

\bibitem[AFR10]{AdinFirerRoichman10}
Ron~M. Adin, Marcelo Firer, and Yuval Roichman, \emph{Triangle-free
  triangulations}, Adv. in Appl. Math. \textbf{45} (2010), no.~1, 77--95.
  \MR{2628788}

\bibitem[Art47]{Artin47}
E.~Artin, \emph{Theory of braids}, Ann. of Math. (2) \textbf{48} (1947),
  101--126. \MR{0019087}

\bibitem[BHMV95]{BlanchetHabeggerMasbaumVogel95}
C.~Blanchet, N.~Habegger, G.~Masbaum, and P.~Vogel, \emph{Topological quantum
  field theories derived from the {K}auffman bracket}, Topology \textbf{34}
  (1995), no.~4, 883--927. \MR{1362791}

\bibitem[Bir74]{Birman74}
Joan~S. Birman, \emph{Braids, links, and mapping class groups}, Princeton
  University Press, Princeton, N.J.; University of Tokyo Press, Tokyo, 1974,
  Annals of Mathematics Studies, No. 82. \MR{0375281}

\bibitem[Jon85]{Jones85}
V.~F.~R. Jones, \emph{A polynomial invariant for knots via von {N}eumann
  algebras}, Bull. Amer. Math. Soc. (N.S.) \textbf{12} (1985), no.~1, 103--111.
  \MR{766964}

\bibitem[Jon87]{Jones87}
\bysame, \emph{Hecke algebra representations of braid groups and link
  polynomials}, Ann. of Math. (2) \textbf{126} (1987), no.~2, 335--388.
  \MR{908150}

\bibitem[Kau87]{Kauffman87}
Louis~H. Kauffman, \emph{State models and the {J}ones polynomial}, Topology
  \textbf{26} (1987), no.~3, 395--407. \MR{899057}

\bibitem[KL94]{KauffmanLins94}
Louis~H. Kauffman and S{\'o}stenes~L. Lins, \emph{Temperley-{L}ieb recoupling
  theory and invariants of {$3$}-manifolds}, Annals of Mathematics Studies,
  vol. 134, Princeton University Press, Princeton, NJ, 1994. \MR{1280463
  (95c:57027)}

\bibitem[Kup96]{Kuperberg96}
Greg Kuperberg, \emph{Spiders for rank {$2$} {L}ie algebras}, Comm. Math. Phys.
  \textbf{180} (1996), no.~1, 109--151. \MR{1403861}

\bibitem[Lic93]{Lickorish93B}
W.~B.~R. Lickorish, \emph{Skeins and handlebodies}, Pacific J. Math.
  \textbf{159} (1993), no.~2, 337--349. \MR{1214075}

\bibitem[Mas17]{Masbaum17}
Gregor Masbaum, \emph{On powers of half-twists in {$M(0,2n)$}}, Glasg. Math. J.
  (2017), 1--6.

\bibitem[OY97]{OhtsukiYamada97}
Tomotada Ohtsuki and Shuji Yamada, \emph{Quantum {${\rm SU}(3)$} invariant of
  {$3$}-manifolds via linear skein theory}, J. Knot Theory Ramifications
  \textbf{6} (1997), no.~3, 373--404. \MR{1457194}

\bibitem[Rob94]{Roberts94}
Justin Roberts, \emph{Skeins and mapping class groups}, Math. Proc. Cambridge
  Philos. Soc. \textbf{115} (1994), no.~1, 53--77. \MR{1253282}

\bibitem[Sty15]{Stylianakis15}
Charalampos Stylianakis, \emph{The normal closure of a power of a half-twist
  has infinite index in the mapping class group of a punctured sphere},
  arXiv:1511.02912 (2015).

\bibitem[Yua17]{Yuasa17b}
Wataru Yuasa, \emph{A {$q$}-series identity via the {$\mathfrak{sl}_3$} colored
  {J}ones polynomials for the {$(2,2m)$}-torus link}, to appear in Proceedings
  of the American Mathematical Society (arXiv:1612.02144) (2017).

\end{thebibliography}
\end{document}